\documentclass[11pt]{article}
\usepackage[utf8]{inputenc}
\usepackage{fullpage}
\usepackage{comment}
\usepackage{amsmath}
\usepackage{amsfonts}
\usepackage{amssymb}
\usepackage{tikz-cd}
\usepackage{color,soul}
\usepackage{array}
\usepackage{zref-savepos}
\usepackage{amsthm}
\usepackage{multirow, bigstrut}
\usepackage[font={small,it}]{caption}
\usepackage{tikz}
\usepackage{bm}
\usetikzlibrary{arrows}

\usepackage{hyperref}
\hypersetup{
    colorlinks=true,
    linkcolor=blue,
    filecolor=magenta,  
    urlcolor=cyan,
}
\usepackage{cleveref}
\usetikzlibrary{calc}
\usepackage{comment}
\usepackage{enumerate}
\usepackage{tikz}
\usepackage{amsmath,accents}
\usetikzlibrary{shapes.geometric}

\newtheorem{theorem}{Theorem}[section]
\newtheorem*{theorem*}{Theorem}
\newtheorem{lemma}[theorem]{Lemma}

\newtheorem{example}[theorem]{Example}
\newtheorem{proposition}[theorem]{Proposition}
\newtheorem{corollary}[theorem]{Corollary}

\theoremstyle{definition}
\newtheorem{definition}[theorem]{Definition}
\newtheorem{remark}[theorem]{Remark}

\theoremstyle{plain}

\newcommand\restr[2]{{
  \left.\kern-\nulldelimiterspace 
  #1 
  \vphantom{\big|} 
  \right|_{#2} 
  }}

\newcommand{\C}{\mathbb{C}}
\newcommand{\R}{\mathbb{R}}
\newcommand{\N}{\mathbb{N}}
\newcommand{\Z}{\mathbb{Z}}

\newcommand{\g}{\mathfrak{g}}
\newcommand{\h}{\mathfrak{h}}
\newcommand{\uu}{\mathfrak{u}}
\newcommand{\Q}{\mathbb{Q}}

\DeclareMathOperator{\Sym}{Sym}

\DeclareMathOperator{\Poly}{Poly}

\DeclareMathOperator{\cpc}{Cap}
\DeclareMathOperator{\Newt}{Newt}
\DeclareMathOperator{\supp}{supp}

\DeclareMathOperator{\trunc}{Trunc}
\DeclareMathOperator{\ttrunc}{trunc}

\DeclareMathOperator{\Deg}{Deg}
\DeclareMathOperator{\CT}{CT}
\DeclareMathOperator{\PS}{LS}

\DeclareMathOperator{\relint}{relint}
\DeclareMathOperator{\character}{char}

\DeclareMathOperator{\sspan}{span}

\DeclareMathOperator{\netflow}{net-flow}

\title{New Bounds for Integer Flows and Verma Modules, via Denormalized Lorentzian Laurent Series}
\author{Jonathan Leake and Maryam Mohammadi Yekta}
\begin{document}
\maketitle
\begin{abstract}
The theory of log-concave polynomials has recently been developed to study objects and problems in combinatorics and other subfields of mathematics. Particular classes of log-concave polynomials called \textit{Lorentzian polynomials} and \textit{denormalized and dually Lorentzian (DL) polynomials} have been used to prove log-concavity statements for various combinatorial sequences. This includes the strongest form of Mason’s log-concavity conjecture on the independent sets of matroids and the log-concavity of sequences of Kostka numbers.


In this paper, we develop an analogous class of power series called \textit{denormalized Lorentzian (DL) Laurent series}. This class is the natural generalization of DL polynomials to homogeneous power series, with the benefit of capturing a number of combinatorial generating series including the Kostant partition function for integer flows of directed graphs. We then analyze specific DL Laurent series to obtain
new bounds for integral flows on general directed acyclic graphs and new bounds for the dimensions of weight spaces of parabolic {$\mathfrak{sl}_{n+1}(\mathbb{C})$} Verma modules.
\end{abstract}
\section{Introduction}

The theories of Lorentzian and log-concave polynomials have been developed over the past 25 years to solve problems and study objects in various subfields of mathematics and computer science (see \cite{Gur09, BH20, ALOV19, ALOGV24}).
Particular classes of log-concave polynomials which has received recent attention is that of \textit{denormalized Lorentzian} (e.g. \cite{BLP23}) and \textit{dually Lorentzian polynomials} \cite{RSW25}. These classes differ from usual Lorentzian polynomials by a simple normalization operation, but this is already enough to capture important combinatorial polynomials, including Schur polynomials, Schubert polynomials, Tutte-like polynomials, covolume polynomials, and conjecturally many more \cite{HMMD22, BEST23, Alu24}.

One major thread of applications of these classes of polynomial lies in approximating and bounding combinatorial quantities via \textit{Gurvits's capacity method} \cite{Gur08}. This method involves bounding specific coefficients of log-concave polynomials via a certain entropy optimization problem based on the polynomial (called the \textit{capacity}). This method has been applied to prove lower bounds for interesting quantities, such as the permanent and mixed discriminant \cite{Gur08,Gur06},
integer flows of graphs \cite{Bar09, Bar12, Gur15, BLP23, LM26},
various quantities related to matroids \cite{AO17,SV17,AOV21}
distribution probabilities related to the traveling salesman problem \cite{KKO21, GKL24}, and beyond.

In this paper, we further develop this thread by defining and applying the capacity method to a new class of log-concave functions: \textit{denormalized Lorentzian (DL) Laurent series}\footnote{We remark that the theory developed here could also be analogously applied to define dually Lorentzian Laurent series. }. This class, motivated by {\S 3.1 of \cite{HMMD22}}, is the natural generalization of denormalized Lorentzian polynomials to homogeneous (two-sided) power series allowing negative powers. The benefit of this definition is that it captures a number of combinatorial generating series, leading to new lower bounds on their coefficients.
This includes new bounds on the number of integral flows on general directed acyclic graphs, and new bounds on the dimensions of weight spaces of parabolic {$\mathfrak{sl}_{n+1}(\C)$} Verma modules. These bounds improve and generalize similar bounds on contingency tables and integer flows obtained in \cite{BLP23,LM26}. We describe our results in more detail now.

\subsection{Main result: Integer flows}

Our first main result gives lower bounds on integer flows of a general directed acyclic graph, based on the capacity of the Kostant partition function. This improves upon the bound achieved in \cite{LM26} which gave a slightly worse and more complicated bound for integer flows (mainly in the complete graph case). As is typical for capacity bounds (since capacity is a convex program up to $\log$-$\log$ transformation, see \Cref{capacity is a convex optimization problem}), this result can also be seen as an algorithmic result for approximating integer flows.

\begin{theorem}
\label{main thm for DAGs}
    Let $G = ([n], E)$ be a directed graph. We say that $i \in [n]$ is a terminal vertex if no edge $i \to j$ exists in $E$. Let $T \subseteq [n]$ be the set of terminal vertices of $G$ and assume $\bm{N} \in \Z^{n}$ satisfies $|\bm{N}|_1 = 0$. For each $i \in [n]$, let $G_i$ be the undirected induced subgraph of $G$ on vertices $\{i, i+1, \cdots, n\}$, and denote by $C_i$ the connected component of $G_i$ containing $i$. Then: 
    \[
    |K_G({\bm{N}})| \ge \cpc_{\bm{N}}(f_G)\cdot \prod_{i = 2}^{n} \max \left\{ \frac{|s_i|^{|s_i|}}{(|s_i| + 1)^{|s_i| + 1}},\frac{|N_i|^{|N_i|}}{(|N_i| + 1)^{|N_i| + 1}} \cdot \delta_{i \in T} \right\}
    \]
    where $s_i = \sum_{j \in C_i} N_j$, and $\delta_{i \in T}$ is the indicator variable of $i$ being a terminal vertex and:
    \[
    \cpc_{\bm{N}}(f_G)  = \inf_{\overset{\bm{x}>\bm{0}}{x_j < x_i \forall i \to j \in E}} \left[ \bm{x^{-N}}\prod_{i \to j \in E}\frac{1}{1 - \frac{x_j}{x_i}} \right]
    .\]
\end{theorem}

Additionally, by considering the complete bipartite directed acyclic graph case, we recover the bounds on contingency tables obtained in \cite{BLP23}. We state this result explicitly now.
\begin{theorem}
\label{main thm for contingency tables}
    Assume $\bm{\alpha} \in \Z_{\ge 0}^n, \bm{\beta} \in \Z_{\ge 0}^m$ are given so that $\bm{\alpha}$ is in decreasing order. Then
    \[
    \CT(\bm{\alpha}, \bm{\beta}) \ge \cpc_{(\bm{\alpha}, -\bm{\beta})}(C).\prod_{i = 1}^m \frac{\beta_i^{\beta_i}}{(\beta_i+1)^{\beta_i+1}} \prod_{i  = 2}^n  \frac{\alpha_i^{\alpha_i}}{(\alpha_i+1)^{\alpha_i + 1}},
    \]
    where
    \[
    \cpc_{(\bm{\alpha}, - \bm{\beta})}(C)  = \inf_{\overset{\bm{x}, \bm{y}>\bm{0}}{x_i < y_j \forall i, j}} \left[ \bm{x}^{-\bm{\alpha}} \bm{y^{\beta}}\prod_{i \in [n],j \in [m]}\frac{1}{1 - \frac{x_i}{y_j}} \right]
    .\]
\end{theorem} 

\subsection{Main result: Parabolic Verma modules}

Beyond the improved bounds on integer flows stated above, we also obtain new bounds on dimensions of weight spaces of parabolic {$\mathfrak{sl}_{n+1} (\C)$} Verma modules. Parabolic Verma modules are well-studied infinite-dimensional representations of Lie algebras, and their weight spaces are a standard decomposition of the underlying infinite-dimensional vector space. Throughout we will denote by $M(\bm\lambda,J)$ the parabolic Verma module of $\mathfrak{sl}_{n+1}(\C)$, where $\bm\lambda$ is the highest weight of the module, and $J$ is a subset of the simple roots of the root system associated to $\mathfrak{sl}_{n+1}(\C)$.

The proofs of our new bounds build on the results of \cite{KMS25}, which study the log-concavity of the characters of parabolic Verma modules. We refer the reader to \cite{KMS25} for further background and discussion of parabolic Verma modules and how they are defined. Our first result is the most general result we have for dimensions of weight spaces of such modules.

\begin{theorem}
Fix $J \subseteq [n]$. Define $0 =: i_0 < i_1 < \cdots < i_r < i_{r+1} := n+1$ so that $J^c = \{i_1, \ldots, i_r\}$. For each integer $t \in [0,r]$, denote by $J_t$ the set $\{i_t + 1, \cdots , i_{t + 1} - 1 \}$.  Construct a complete directed multipartite graph $G_J$ with $r+1$ parts, the $t$-th part being $J_t \cup \{i_t\}$ for $t \in [0,r]$, where edges are directed $u \to v$ when $u < v$.
Let $M(\bm \lambda,J)$ be a parabolic Verma module and denote by $\lambda_t$ the partition $(\lambda_{i} - \lambda_{1 + \max J_t}: i \in J_t)$. Then for any weight $\bm{\mu}\in \Z^{n+1}$:
    \[
        \dim M(\bm{\lambda}, J)_{\bm{\mu}} \ge \cpc_{\bm{\mu}}\left(\character M (\bm{\lambda}, J) \right) \prod_{i = 1}^{n}\frac{|m_i|^{|m_i|}}{(|m_i| + 1)^{|m_i| + 1}},
    \]
where:
\[\character M(\bm{\lambda}, J) = f_{G_J} (x_i: i \in [n+1]) \cdot \prod_{ t =0}^r s_{\lambda_t}(x_i: i \in J_t). \]
$s_{\lambda_t}$ is the Schur polynomial indexed by $\lambda_t$ and  $m_i = -\sum_{k \in C_i} \mu_k$ where $C_i$ is the connected component containing $i$ in the undirected induced subgraph of $G_J$ on $\{i, \cdots, n+1\}$.

\end{theorem}

Finally, it is also desirable to have bounds which do not depend on the capacity optimization problem. Thus we apply the dual program techniques of \cite[Lemma 5]{Bar12}, \cite[Proposition 6.2]{BLP23} and \cite[Proposition 3.1]{LM26} to obtain explicit bounds as corollaries of the above results. The bounds are general, and thus probably can be improved in specific cases. We discuss this further in Section \S \ref{subsections: explicit bounds for dimensions of weight spaces}.

\begin{corollary}
Suppose  $M(\bm{\lambda}, J)$ is a parabolic Verma module. Fix the following
\begin{itemize}
 
    \item $\bm{\nu} \in \Z^{[n+1]}_{\ge 0}$ with $\bm{\nu}_t := \restr{\bm{\nu}}{J_t}$ such that $|\bm{\nu}_t|_1 = |{\lambda}_t|_1$ and $\nu_{i_{t+1}} = 0$ for all $t \in \{0,1,\ldots,r\}$, and
    \item a (not necessarily integral) flow $\phi$ of $G_J$ with net-flows $\bm\mu-\bm\nu$.
\end{itemize}

Then
    \begin{align*}
        \dim M(\bm{\lambda}, J)_{\bm{\mu}} &\ge   \prod_{e \in E(G_J)}\frac{(\phi(e) + 1)^{\phi(e) + 1}}{\phi(e)^{\phi(e)}} \cdot \prod_{ t = 0}^r K_{\lambda_t, \bm{\nu}_t} \cdot \prod_{i = 1}^{n}\frac{|m_i|^{|m_i|}}{(|m_i| + 1)^{|m_i| + 1}} 
    \end{align*}
    where $m_i = -\sum_{k \in C_i} \mu_k$, and $K_{\lambda_t, \bm{\nu}_t}$ is the Kostka number indexed by $\lambda_t, \bm{\nu}_t$. 
\end{corollary}
\subsection{Technical contribution}

The bounds of this paper are derived from relating the coefficients of various log-concave polynomials and power series to the capacity optimization problem defined above. This is a technique that has at this point a relatively established history, dating back to its first use by Gurvits around 2004. One of the main contributions of this paper is then to extend this technique to power series in a systematic way. A new feature of the analysis is then the domain of convergence of the input power series. Specifically, we need to know how log-concavity properties of power series relate to their domains of convergence.

To give a sense of this relation, let us consider a standard univariate power series
\[
    p(x) = \sum_{n=0}^\infty p_n x^n,
\]
and let us assume that the coefficients of $p$ are positive and log-concave (i.e., $p_n^2 \geq p_{n-1} p_{n+1}$ for all $n$). For such a $p$, we can obtain the bound
\[
    \cpc_k(p) \geq p_k \geq \frac{k^k}{(k+1)^{k+1}} \cdot \cpc_k(p) \qquad \text{where} \qquad \cpc_k(p) = \inf_{x > 0} \frac{p(x)}{x^k}
\]
for all $k$ (see Corollary \ref{a bound for coefficients of bivariate DL Laurent series}\footnote{Actually we consider a certain homogenization of $p$, but our results are equivalent to the ones we describe here.}). The upper bound is straightforward and only requires positivity of the coefficients of $p$, whereas the lower bound is more difficult and relies heavily on the log-concavity assumption.

Notice that we make no explicit assumption here about convergence of $p(x)$, which of course is needed to make the above bounds sensible. This is because a non-empty region of convergence is guaranteed by the log-concavity property. To see this, note that positivity and log-concavity imply
\[
    \frac{p_1}{p_0} \geq \frac{p_2}{p_1} \geq \frac{p_3}{p_2} \geq \cdots \geq \frac{p_n}{p_{n-1}} \geq \frac{p_{n+1}}{p_n} \geq \frac{p_{n+2}}{p_{n+1}} \geq \cdots.
\]
By the ratio test for convergence, we then have that $p(x)$ converges whenever
\[
    \lim_{n \to \infty} \frac{p_{n+1} x^{n+1}}{p_n x^n} \leq x \cdot \frac{p_1}{p_0}
\]
is strictly less than 1. Thus we are guaranteed $\frac{p(x)}{x^k}$ is finite for $x < \frac{p_0}{p_1}$, and therefore $\cpc_k(p)$ is finite as well.

The main technical contribution of this paper is then to extend this line of reasoning to multivariate power series (including Laurent series with possibly infinitely many negative powers as well). This requires a notion of log-concavity for multivariate power series, which we define and develop as a natural generalization of such notions for multivariate polynomials. We then need to extend previous bounds to the power series setting, which requires analysis of regions of convergence for multivariate power series with the aforementioned log-concavity properties. Once this is done, we obtain our main results by applying our technical results to particular power series which encode combinatorial data for the various objects we want to study.

\subsection{A small example}

To give a better sense of our bounds, we also show how we can use our bounds to obtain a weak version of Stirling's approximation. As mentioned above, our bounds imply for $p(x) = \sum_{n=0}^\infty p_n x^n$ where $(p_n)_{n=0}^\infty$ is positive and log-concave that
\[
    \cpc_k(p) \geq p_k \geq \frac{k^k}{(k+1)^{k+1}} \cdot \cpc_k(p) \qquad \text{where} \qquad \cpc_k(p) = \inf_{x > 0} \frac{p(x)}{x^k}.
\]
If $p(x) = e^x = \sum_{n=0}^\infty \frac{x^n}{n!}$, then $(p_n)_{n=0}^\infty$ is positive and log-concave. Further, basic calculus implies
\[
    \cpc_k(p) = \inf_{x > 0} \frac{e^x}{x^k} = \left[\frac{e^x}{x^k}\right](k) = \frac{e^k}{k^k}.
\]
Thus the above inequalities imply
\[
    \frac{e^k}{k^k} \geq \frac{1}{k!} \geq \frac{e^k}{(k+1)^{k+1}} \iff \left(\frac{k}{e}\right)^k \leq k! \leq (k+1) \cdot \left(\frac{k+1}{e}\right)^k
\]
for all $k$. As claimed, this is weak version of the fact that $k! \approx \sqrt{2\pi k} \left(\frac{k}{e}\right)^k$.

We note two main differences between this weak Stirling's approximation and our main results. First, we usually cannot do the calculus step above to actually compute the capacity. So our results are either in terms of the capacity, or in terms of some lower bound on the capacity we can obtain using convex analytic techniques (see \Cref{subsection: explicit bounds for integer flows} and \Cref{subsections: explicit bounds for dimensions of weight spaces}). And second, the power series we analyze to obtain the main results have many variables instead of just one. Generally, the approximation factors one can achieve decay exponentially in the number of variables, and thus the bounds we obtain appear much worse than the one obtained here. That said, this exponential decay is expected to be necessary in general, and thus we expect our bounds to be close to tight.

\subsection{Paper structure}

We start by reviewing some previous works on log-concave polynomials and discussing previously known bounds for contingency tables and type A Kostant partition functions in \S \ref{Section: background}. We then introduce and study denormalized Lorentzian Laurent series in \S \ref{Section: DL Laurent series}. We continue studying these Laurent series and their domains of convergence to prove bounds on their coefficients using the capacity function in \S \ref{section: capacity for Laurent series}. We then introduce flow Laurent series and present lower bounds for the number of integer flows of a general directed acyclic graph in \S \ref{section: Flows of a graph}. We will use some convex analysis tools in the same section to obtain more explicit lower bounds for these numbers. We then restrict to two special graphs to obtain an improved bound on type A Kostant partition functions and re-prove a bound on contingency tables. Lastly, we apply our bounds to the characters of parabolic Verma modules and obtain bounds on dimensions of weight spaces of parabolic Verma modules in \S \ref{section: applying bounds to parabolic verma modules}.

\section{Background}
\label{Section: background}
\subsection{Basic notation}
We use $\C$ and $\R$ to denote the set of complex and real numbers. Let $\R^* = \R \cup \{ \infty\}$ and $\N = \{ 0, 1, \cdots, \}$. For any two integers $m > n \ge 1$, define $[n] := \{ 1, \cdots, n\}$, $[n,m] := \{ n, n + 1, \cdots, m\}$ and $\partial_n =\frac{\partial}{\partial x_n}$. 

For two vectors $\bm{x} = (x_1, \cdots, x_n)$ and $\bm{\alpha} = (\alpha_1, \cdots, \alpha_n)$, $\bm{x^{\alpha}}$ is the product $x_1^{\alpha_1} \cdots x_n^{\alpha_n}$ and $\bm{\alpha}!$ is $\prod_{i=1}^n \alpha_i!$ whenever $\bm{\alpha} \in \N^n$. We say that $\bm{x}\ge \bm{\alpha}$ if $x_i \ge \alpha_i$ for all $i \in [n]$. Moreover, $|\bm{x}|_1 = \sum_{i = 1}^n x_i$.

The support of a polynomial $p \in \R[x_1, \cdots, x_n]$ is the set $\supp(p) = \{ \bm{\alpha} \in \Z^n : p_{\bm{\alpha}} \neq 0 \}$, and $\deg_{x_i}(p)$ is the degree of the variable $x_i$ in $p$. 

For $\bm{\mu} \in \Z^n$, let $\R^{\bm{\mu}}[x_1, \cdots, x_n] := \{p(\bm{x}) \in \R[\bm{x}]: \deg_{x_i}p \le \mu_i \}$. A linear operator $T: \R^{\bm{\mu}}[x_1, \cdots, x_n] \to \R[x_1, \cdots, x_m]$ is said to be homogeneous if there exists some $k \in \N$ such that $T[\bm{x^{\alpha}}]$ is a homogeneous polynomial of degree $|\bm{\alpha}|_1 + k$ for any $\bm{\alpha} \in \N^n$.

A nonempty set $A \in \ \Z^n$ is {$M$-convex} if for any $\bm{x,y} \in A$ and $i \in [n]$ such that $x_i > y_i$, there exists some $j \in [n]$ satisfying $x_j < y_j$ and $\bm{x} + \bm{e}_j - \bm{e}_i, \bm{y} - \bm{e_j} + \bm{e_i}\in A$ 

\subsection{Lorentzian polynomials}

Lorentzian polynomials (in their various equivalent forms) were developed in \cite{Gur09,ALOV19,BH20}. We will review some of their basic properties now.

\begin{definition}[{\cite[Definition 2.1]{BH20}}, Lorentzian polynomials]
    \label{definition of Lorentzian polynomials}
    A homogeneous polynomial $p \in \R_{\ge 0}[x_1,\cdots, x_n]$ is \textit{strictly Lorentzian} if it has positive coefficients and for any $1 \le i_1, \cdots, i_{d-2} \le n $ and $\bm{x} \in \R^n$, the  Hessian of $\partial_{i_1} \cdots \partial_{i_n}p$ at $\bm{x}$ is nonsingular and has exactly one positive eigenvalue. The closure of the set of strictly Lorentzian polynomials (in the Euclidean space of homogeneous polynomials of degree at most $d$ with $n$ variables and non-negative coefficients) is called the set of \textit{Lorentzian} polynomials.
\end{definition}
Note that by definition, $\partial_i p$ is Lorentzian if $p$ is a Lorentzian polynomial. 

\begin{lemma} [{\cite[Example 2.26]{BH20}}]
\label{bivariate homog polynomial is Lorentzian iff it has ULC coeffs}
    A bivariate homogeneous polynomial $p(x,y)= \sum_{k = 0}^d a_k x^k y^{d-k} $ is Lorentzian if and only if $\{ a_k\}_{k = 0}^d$ is an ultra log-concave sequence, i.e.,it doesn't have any internal zeroes and for any valid index $i$, we have \[\left(\frac{a_i}{{n \choose i}}\right)^2 \ge \frac{a_{i-1}}{{n \choose i-1}}\cdot \frac{a_{i+1}}{{n \choose i+1}}\].
\end{lemma}
We say that an operator $T:\R^{\bm{\mu}}[x_1, \cdots, x_n] \to \R[x_1, \cdots, x_m]$ preserves the class of Lorentzian polynomials if $T[P]$ is Lorentzian for any Lorentzian polynomial $p$. The following theorem gives a sufficient condition for $T$ to preserve the class of Lorentzian polynomials. 

\begin{theorem}[{\cite[Theorem 3.2]{BH20}}]
\label{symbol theorem for Lorentzian}
    If $T : \R^{\bm{\mu}}[x_1, \cdots,x_n] \to \R[x_1, \cdots, x_m]$ is a linear homogeneous operator, then $T$ preserves the class of Lorentzian polynomials if $\Sym_T(\bm{x}, \bm{z})$ is a Lorentzian polynomial, where:
    \[
    \Sym_T(\bm{x},\bm{z}) := T_{\bm{x}} \left[\prod_{i = 1}^n ( x_i + z_i)^{\mu_i} \right].
    \]
    $\Sym_T$ is called the symbol of $T$, and $T_{\bm{x}}$ treats all the $z$ variables as constants and only acts on the $x$ variables. 
\end{theorem}
Finally, note that Lorentzian polynomials and $M$-convex sets are deeply connected: 
\begin{theorem} [{\cite[Theorem 3.10]{BH20}}] 
\label{M-convexity and Lorentzian}
Let the generating function of a set $J \in \Z_{\ge 0}^n$ be the polynomial
\[
f_J(x_1, \cdots, x_n) = \sum_{\bm{\alpha} \in J} \frac{\bm{x^{\alpha}}}{\bm{\alpha}!}.
\]
Then $J$ is an $M$-convex set if and only if $f_J$ is Lorentzian. Moreover, the {support} of any Lorentzian polynomial $p \in \R_{\ge 0}[x_1, \cdots, x_n]$ is an M-convex set.
\end{theorem}
\subsection{Denormalized Lorentzian polynomials}
Let $N : \R[x_1, \cdots, x_n] \to \R[x_1, \cdots, x_n]$ be the \textit{normalization} operator, defined by $N[\sum_{\bm{\alpha}} p_{\bm{\alpha}} \bm{x^{\alpha}}] = \sum_{\bm{\alpha}}p_{\bm{\alpha}}\frac{\bm{x^{\alpha}}}{\bm{\alpha}!}$. 
A polynomial $p(\bm{x})$ is said to be denormalized Lorentzian (DL) if $N[p]$ is a Lorentzian polynomial. 

\begin{remark}
    \label{DL polynomials are closed}
    
Since normalization is a continuous operator, it follows that the class of DL polynomials of degree at most $d$ in $n$ variables is closed in the Euclidean space of homogeneous polynomials of degree at most $d$ with $n$ variables. 

\end{remark}
\begin{corollary}
\label{bivariate homog polynomial is DL iff it has LC coeffs}
    It follows from definition and \Cref{bivariate homog polynomial is Lorentzian iff it has ULC coeffs} that a bivariate homogeneous polynomial $p(x,y)= \sum_{k = 0}^d a_k x^k y^{d-k} $ is DL if and only if $\{ a_k\}_{k = 0}^d$ is a log-concave sequence, i.e., it doesn't have any internal zeroes and for any valid index $i$, we have $a_i^2 \ge a_{i-1}\cdot a_{i + 1}$. 
\end{corollary}

\begin{theorem}[{\cite[Theorem 3]{HMMD22}}]
    Schur polynomials are DL polynomials. 
\end{theorem}
We will review some useful operations that preserve the class of DL polynomials. 
\begin{theorem}[{\cite[Lemma 4.8]{BLP23}}]
\label{DL polynomials are closed under scaling variables}
If $p(x_1, \cdots, x_n)$ is a DL polynomial and $y_1, \cdots, y_n \in \R_{>0}$, then $p(y_1 x_1, \cdots, y_n x_n)$ is a DL polynomial as well.
    
\end{theorem}

\begin{theorem}[{\cite[Lemma 4.8]{BLP23}}]
\label{x1=x2 polynomials}
    Let $T_{x_1=x_2}^{\Poly}: \R[x_1, \cdots, x_n] \to \R[x_1, \cdots, x_m]$ be the linear operator defined as $T_{x_1=x_2}^{\Poly}[p(x_1, x_2, x_3,\cdots, x_n)] = p(x_1, x_1, x_3, \cdots, x_n)$. Then $T^{\Poly}_{x_1 = x_2}$ preserves the class of DL polynomials.
\end{theorem}

\begin{theorem} [{\cite[Corollary 3.8]{BH20}}] 
\label{multiplication of DL polynomials stays DL}
The class of DL polynomials is closed under multiplication. 
\end{theorem}

\begin{corollary}
\label{multiple of DL polynomials}
    Let $t \in \N$ and $p,q \in \R[x_1,\cdots
    , x_n]$ be polynomials satisfting $p(\bm{x}) = x_1^t q(\bm{x})$. Then $p(\bm{x})$ is DL if and only if $q(\bm{x})$ is DL. 
\end{corollary}
\begin{proof}
    If $q(\bm{x})$ is a DL polynomial, then $p(\bm{x})$ is also DL using \Cref{multiplication of DL polynomials stays DL}. Conversely, if $p(\bm{x})$ is DL, then:
    \[
    N[q(\bm{x})] = N[x_1^{-t}p(\bm{x})] = \partial_1^t N[p(\bm{x})] 
    \]
    and $N[q(\bm{x})]$ is Lorentzian since the class of Lorentzian polynomials is closed under derivation. 
\end{proof}

\begin{theorem}
\label{truncation below preserves DL polynomials}
    For any $t \in \Z$ and $i \in [n]$, let $\trunc_{x_i}^{t} : \R[x_1, \cdots, x_n] \to \R[x_1, \cdots, x_n]$ be a linear operator defined as follows:
    \[ \trunc_{x_i}^t [\bm{x}^{\bm{\alpha}}] = \begin{cases}
        x_i^{-t} \bm{x^{\alpha}} & \text{ if } \hspace{0.2cm} \alpha_i \ge t\\
        0 & \text{ otherwise}
    \end{cases}\]
    Then $\trunc_k^i$ preserves the class of DL polynomials. 
\end{theorem}
\begin{proof}
   Similar to the proof of \Cref{multiple of DL polynomials}, we have that:
   \[
   N[\trunc_{x_i}^t[\bm{x^{\alpha}}]] = \partial_i^t N[\bm{x^{\alpha}}]
   \]
   and therefore $N[\trunc_{x_i}^{t}[p]] = \partial_i^t N[p]$ is Lorentzian whenever $N[p]$ is Lorentzian.
\end{proof}

\begin{theorem}
\label{truncation above preserves DL polynomials}
    For any $t \in \Z, i \in [n]$, let $\ttrunc_{x_i}^t : \R[x_1, \cdots, x_n] \to \R[x_1, \cdots, x_n]$ be a linear operator defined as follows:
    \[ \ttrunc_{x_i}^t [\bm{x}^{\bm{\alpha}}] = \begin{cases}
        \bm{x^{\alpha}} & \text{ if } \hspace{0.2cm} \alpha_i \le t\\
        0 & \text{ otherwise}
    \end{cases}\]
    Then $\ttrunc_k^i$ preserves the class of DL polynomials. 
\end{theorem}
\begin{proof}
    Without loss of generality, assume that $i = 1$. Let $S:= N \circ \ttrunc_{x_i}^t \circ N^{-1}$. We just need to prove that $S$ preserves the class of Lorentzian polynomials. It suffices to show that $\restr{S}{\R^{\bm{\mu}}[\bm{x}]}$ preserves the class of DL polynomials for any $\bm{\mu}$. Note that $S$ is a homogeneous operator, and by \Cref{symbol theorem for Lorentzian}, we only need to show that the symbol of $S$ is Lorentzian. 
    \begin{align*}
    \Sym_{\restr{S}{\R^{\bm{\mu}}[\bm{x}]}}(\bm{x}, \bm{z}) &= \sum_{\bm{0} \le \bm{\lambda} \le \bm{\mu}} {\bm{\mu} \choose \bm{\lambda}} \bm{z}^{\bm{\mu- \lambda}}S[\bm{x}^{\bm{\lambda}}]\\
    &= (\bm{\mu}!)^n \sum_{\overset{0 \le \bm{\lambda} \le \bm{\mu}}{\lambda_1 \le t}} \frac{\bm{z}^{\bm{\mu - \lambda}}}{(\bm{\mu - \lambda})!}.\frac{\bm{x^{\lambda}}}{(\bm{\lambda})!}
    \end{align*}
    which is the generating polynomial of an M-convex set up to a scalar, and therefore, is Lorentzian by \Cref{M-convexity and Lorentzian}.
\end{proof}

Br\"and\'en, Leake and Pak use the capacity function (originally defined by Gurvits \cite{Gur06}) to obtain a lower bound for the coefficients of DL polynomials \cite{BLP23}, and utilize it to prove a lower bound for the number of contingency tables with marginals $\bm{\alpha}, \bm{\beta}$. We will go over the definition of contingency table and their bounds in \S \ref{subsection: contingency tables}. For now, let us state \cite{BLP23}'s capacity bound for the coefficients of a DL polynomial.
\begin{theorem}[{\cite[Theorem 5.10]{BLP23}}]
\label{theorem from BLP}
    Let $\bm{\alpha} \in \Z_{\ge 0}^n$ and $p(x_1, \cdots, x_n) = \sum_{\bm{\mu} \in \N^n} p_{\bm{\mu}}\bm{x^\bm{\mu}}$ be a DL polynomial of degree $d$. Let $d_i$ be the degree of $x_i$ in $\restr{\partial_{i+1}^{\alpha_{i+1}}\cdots \partial_n^{\alpha_n}p}{x_{i+1} = \cdots = x_n = 0}$, and let $d_n$ be the degree of $x_n$ in $p$. Then:
    \[
    [\bm{x^{\alpha}}]p \ge \cpc_{\bm{\alpha}}(p). \prod_{i = 2}^n \max\left \{ \frac{\alpha_i^{\alpha_i}}{(\alpha_i + 1)^{\alpha_i + 1}}, \frac{(d_i - \alpha_i)^{d_i - \alpha_i}}{(d_i - \alpha_i + 1)^{d_i - \alpha_i + 1}}\right\},
    \]
    where $\cpc_{\bm{\alpha}}(p) = \inf_{\bm{x>0}}\frac{f(\bm{x})}{\bm{x^{\alpha}}}$ is the {capacity} of $p$ at point $\bm{\alpha}$.
\end{theorem}

\begin{remark}[{\cite[Remark 11.2]{BLP23}}]
    \label{capacity is a convex optimization problem}
    For a polynomial $p \in \R_{\ge 0}[x_1, \cdots, x_n]$, we can interpret $\cpc_{\bm{\alpha}}(p)$ as a convex polynomial optimization problem:
    \[
    - \log \cpc_{\bm{\alpha}}(p) = \inf_{\bm{y} \in \R^n} \left[ -\log p(e^{\bm{y}}) + \left< \bm{y},\bm{\alpha}\right>\right],
    \]
    where $\log$ is the natural logarithm function and $\langle \cdot, \cdot \rangle$ is the usual dot product. Note that both $- \log$ and the dot product are convex functions, and $\R^n$ is a convex set. 
\end{remark}
\subsection{A Lower bound for the number of contingency tables and type A Kostant partition functions}
\label{subsection: contingency tables} 

For $\bm{\alpha} \in \Z^n, \bm{\beta} \in \Z^m$, a contingency table of marginals $(\bm{\alpha}, \bm{\beta})$ is a matrix $A = (A_{ij}) \in \Z_{\ge 0}^{n \times m}$, with entries in the $i$-th row summing up to $\alpha_i$ and entries in the $j$-th row summing to $\beta_j$. Contingency tables are often studied as lattice points in transportation polytopes \cite{DK14}.
Let $K = ( K_{ij})_{i,j}$ be an $n \times m$ matrix with entries in $\Z_{\ge 0} \cup \{\infty  \}$. Then we say that $A$ is a $K$-contingency table with marginals $(\bm{\alpha}, \bm{\beta})$ if $A$ is a contingency table with marginals $(\bm{\alpha}, \bm{\beta})$ and $A_{ij} \le K_{ij}$ for all $i\in [n],j \in [m]$. Let $CT_{K}(\bm{\alpha}, \bm{\beta})$ be the number of $K$-contingency tables with marginals $(\bm{\alpha}, \bm{\beta})$. If $K_{ij} = \infty$ for all $i,j$, we refer to $\CT_{K}(\bm{\alpha}, \bm{\beta})$ by $\CT(\bm{\alpha}, \bm{\beta})$ for simplicity. We note that the problem of finding $\CT_K(\bm{\alpha}, \bm{\beta})$ is \#P-complete, even if we restrict to the case $n = 2$ (\cite[Theorem 1]{DKM97}).  

For an $n \times m$ finite matrix $K$ with column marginals $\bm{\gamma}$, consider the following polynomial: 
\[
{P}_K(x_1, \cdots, x_n; y_1, \cdots, y_m) = \prod_{i = 1}^n \prod_{j = 1}^m \sum_{l = 0}^{K_{ij}} x_i^l y_j^{l} = \sum_{\bm{\alpha, \beta}} CT_K(\bm{\alpha}, \bm{\beta}) \bm{x^{\alpha}y^{\beta}}.
\]
Modify $P_K$ as follows to get a homogeneous polynomial:
\[
\overset{\sim}{P}_K(x_1, \cdots, x_n; y_1, \cdots, y_m) = \prod_{i = 1}^n \prod_{j = 1}^m \sum_{l = 0}^{K_{ij}} x_i^l y_j^{K_{ij}-l} = \sum_{\bm{\alpha, \beta}} CT_K(\bm{\alpha}, \bm{\beta}) \bm{x^{\alpha}y^{\gamma-\beta}}.
\]
$\overset{\sim}{P}_K(\bm{x},\bm{y})$ is a DL polynomial by \Cref{multiplication of DL polynomials stays DL} and \Cref{bivariate homog polynomial is DL iff it has LC coeffs}. \Cref{theorem from BLP} then gives a lower bound for $\CT_{K}(\bm{\alpha},\bm{\beta})$ whenever $K$ is finite. If $K$ is not finite however, authors of \cite{BLP23} write $K$ as the limit point of a sequence $\{ K_t\}_{t \in \N}$ of finite matrices to get the following general theorem:
\begin{theorem}[{\cite[Theorem 2.1]{BLP23}}]
    \label{bound for contingency tables in BLP}
    For $\bm{\alpha} \in \N^n, \bm{\beta} \in \N^m$ and $K \in (\Z_{\ge 0} \cup \{ \infty\})^{n \times m}$ we have: 
    \[
    CT_{K}(\bm{\alpha}, \bm{\beta}) \ge \cpc_{(\bm{\alpha}, \bm{\beta})}(P_K). \prod_{i = 2}^n \frac{\alpha_i^{\alpha_i}}{(1 + \alpha_i)^{1 + \alpha_i}} \prod_{j = 1}^m \frac{\beta_j^{\beta_j}}{(1 + \beta_j)^{1 + \beta_j}},
    \]
    where:
    \[
    P_K = \prod_{i = 1}^n \prod_{j = 1}^m \sum_{l = 0}^{K_{ij}} x_i^l y_j^l.
    \]
\end{theorem}

Integral flows on complete graphs are closely related to contingency tables.
This connection is employed in \cite{LM26} to obtain a lower bound on the number of integral flows of such graphs, also known as the \textit{Kostant partition number}. Kostant partition numbers are used to express other important quantities in representation theory, such as Kostka numbers and the Littlewood-Richardson coefficients \cite{Hum08}, see also \S \ref{subsection: verma modules and parabolic verma modules}. 

Let $G = ([n], E)$ be an acyclic directed graph, and let $\bm{N} \in \Z_{\ge 0}^n$. An $\bm{N}$-flow of $G$ is a $|E|$-tuple $\bm{\Phi} = (\Phi_e)_{e \in E} \in \N^{E}$ satisfying:
\[
\netflow_{\Phi}(v) :=\sum_{e : u \to v} \Phi_e - \sum_{e: v \to w} \Phi_e= N_v \hspace{1 cm} \forall v \in [n].
\]
Let $K_{n+1}$ be the complete acyclic directed graph with vertices $0, 1, \cdots , n$, and edges directed from $j$ to $i$ for $j > i$. Let $\bm{N} = (N_0, N_1 , \cdots , N_{n-1}, - \sum_{i = 1}^{n-1} N_i)\in \Z^{n+1}$. Denote by $K_n(\bm{N})$ the number of $\bm{N}$-flows of $K_{n+1}$. For certain $\bm{\alpha}, \bm{\beta} \in \Z_{\ge 0}^n $ and $ K \in \Z_{\ge 0}^{n \times n}$, the set of $\bm{N}$-flows of $K_{n+1} $ are in bijection with the set of $K$-contingency tables with marginals $(\bm{\alpha},\bm{\beta})$. See \cite[Section 1.3]{MMR17} and \cite[Section 2.1]{LM26} for more details and proof of this bijection. \Cref{bound for contingency tables in BLP} for $K$-contingency tables with marginals $(\bm{\alpha},\bm{\beta})$ implies: 



\begin{theorem}[{\cite[Theorem 2.28]{LM26}}]
\label{lower bounds for integral flows in LM}
Let $\bm{N} = (N_0, \cdots, N_n) \in \Z^{n+1}$, and $s_i = \sum_{k = 0}^i N_i$ for $0 \le i \le n-1$, and let $\bm{\alpha} = (s_0, \cdots, s_{n-1})$ and $\bm{\beta} = (s_{n-1}, \cdots, s_0)$. We have: 

\[
K_n(\bm{N}) \ge \cpc_{\bm{\alpha}, \bm{\beta}}(\Phi). \max_{0 \le i \le n-1}\left\{ \frac{(s_i + 1)^{s_i + 1}}{s_i^{s_i}} \right\} \prod_{i = 0}^{n-1} \left[\frac{s_i^{s_i}}{(s_i + 1)^{s_i + 1}}\right]^2 ,
\]
where 
\[
\Phi(\bm{x}, \bm{y}) = \prod_{\substack{ 0 \le i,j \le n-1 \\ i + j \le n} } \frac{1}{1- x_i y_j}.
\]
\end{theorem}


Note that \Cref{lower bounds for integral flows in LM} uses a bijection between the integral flows of a complete graph and $K$-contingency tables of a given marginal. A similar bijection can be constructed for any arbitrary acyclic graph by setting the entries of $K$ corresponding to the missing edges in our graph to be zero. So with some work, one could prove a result similar to \Cref{lower bounds for integral flows in LM} for any acyclic directed graph. 

\subsection{Verma modules and parabolic Verma modules}
\label{subsection: verma modules and parabolic verma modules}

In this subsection, we will briefly discuss some terminology related to semi-simple algebras, and two natural classes of modules over them, namely Verma modules and parabolic Verma modules. We refer to \cite{Hum08} and \cite[\S 2]{KMS25} for a more detailed study of these topics.

Let $\mathfrak{g} = \mathfrak{sl}_{n+1}(\C)$ for some $n \in \N$, let $\mathfrak{h}$ be the sub-algebra of traceless diagonal matrices, also known as the Cartan sub-algebra of $\mathfrak{g}$, and let $\h^*$ be the dual space of this sub-algebra. Note that we can think of any $\bm{\nu} \in \h^*$ as a vector in $\C^{n+1}$ with $|\bm{\nu}|_1 = 0$. For any $\bm{\lambda} \in \h^*$, the Verma module indexed by $\bm{\lambda}$ is some quotient of the universal enveloping of $\g$. We refrain from explaining more details about this definition, as it is out of the scope of this paper. 

Let $\epsilon_i$ be the function that sends a matrix $H = [h_{i,j}]_{i,j \in[n]}$ to $h_{i,i} \in \C$. We denote by $\Delta$ the following set of functions:
\[
\Delta = \{ \epsilon_i - \epsilon_j: i < j \in [n+1] \}
\]
$\Delta$ is called a set of \textit{positive roots} of $\g$. For any $\bm{\mu} \in \h^*$, the \textit{weight space multiplicity} $\dim M({\bm{\lambda}})_{\bm{\mu}}$ is the number of ways of writing $\bm{\lambda} - \bm{\mu}$   as a sum of positive roots, which in this case, is equal to $K_{n}(\bm{\lambda} - \bm{\mu})$. Define \textit{the character} of $M(\bm{\lambda})$ denoted by $\character M(\bm{\lambda})$ to be the generating series $\sum_{\bm{\mu} \in \h^*} \dim M(\bm{\lambda})_{\bm{\mu}}\bm{x^{-\mu}}$. Then $\character M(\bm{\lambda})$ is the following Laurent series:
\[
\character M(\bm{\lambda}) (x_1, \cdots, x_{n+1}) = \bm{x^{\lambda}} \cdot \prod_{\epsilon_i - \epsilon_j \in \Delta}\left( 1 + x_jx_i^{-1} + x_j^2x_i^{-2} + \cdots \right) 
\]
It is shown in \cite[Proposition 13]{HMMD22} that the polynomial part of any shift of $\character M(\lambda)$ is a DL polynomial, a property that we will later call denormalized Lorentzianity for Laurent series (see \S \ref{Section: DL Laurent series}).

Given $J \subseteq [n]$, define the \textit{$J$-dominant integral weights} to be the following subset of $\h^*$:
\[
\Lambda^+_J = \{ \bm{\lambda} \in \h^*: \bm\lambda(H_i) \in \Z_{>0} \hspace{.2 cm} \forall i \in J \},
\]
where $H_i \in \h$ is the matrix with a 1 in its $i$-th diagonal entry and a $-1$ in its $(i+1)$-th diagonal. For each $\bm{\lambda} \in \Lambda_J^+$, the parabolic Verma module $M(\bm{\lambda}, J)$, is a quotient of the Verma module $M(\bm{\lambda})$.

Each semi-simple lie algebra has a \textit{Dynkin diagram} associated with it. It is a well known fact that the Dynkin diagram of $\g$ is simply a path with vertices $[n]$ with $i$ connected to $i + 1$ for $i \in [n-1]$. Consider the induced subgraph of this Dynkin diagram on vertices $J \subseteq [n]$, and partition it into its connected components $J_0 \sqcup \cdots \sqcup J_l$. Note that $J_0, \cdots, J_l$ partition $J$ into maximal contiguous intervals. 
Furthermore, construct a graph $G_J$ on vertices $[n+1]$ and include the edge $i \to j$ if $i < j$ and $[i,j-1]\not \subseteq J$. We will explore the properties of this graph in \S \ref{section: applying bounds to parabolic verma modules}. The following identity is proven in \cite[\S 3]{KMS25}:
\[
\character M(\bm{\lambda},J) = \prod_{\overset{i < j \in [n+1]}{i \to j \in E(G)}} \left( 1 + x_jx_i^{-1} + x_j^2 x_i^{-2} + \cdots\right) \cdot \prod_{t = 0}^l s_{\lambda_t}(x_i: i \in J_r)
\]
where $s_{\lambda_t}$ is the Schur polynomial indexed by the composition $\lambda_t := (\lambda_i - \lambda_{1 + \max{J_t}} : i \in J_t)$, and $\max J_t$ is the largest element of $J_t$. One could verify that $\lambda_t$ is a partition since $\bm\lambda \in \Lambda_J^+$. This identity implies that the polynomial part of any shift of $\character M(\bm{\lambda}, J)$ is DL \cite[Theorem 1.5]{KMS25}, or by our terminology, that $\character M(\bm{\lambda}, J)$ is a DL Laurent series. We will use our tools to prove a lower bound for $\dim M(\bm{\lambda},J)_{\bm{\mu}}$.

\section{Denormalized Lorentzian Laurent series}
\label{Section: DL Laurent series}
In this section, we define the class of denormalized Lorentzian (DL) Laurent series. We explore some of their basic properties, and try to draw analogs between this class of Laurent series and DL polynomials. We aim to eventually prove lower bounds for the coefficients of these Laurent series, and use them to derive bounds for integer flows and dimensions of weight spaces of parabolic Verma module. Recall that in the context of this paper, a Laurent series is an infinite sum $\sum_{\bm{\alpha} \in \Z^n} p_{\bm{\alpha}} \bm{x^{\alpha}}$.
\subsection{Definition and basic properties}
\label{definition and basic properties}

Let $\Poly : \R((x_1, \cdots, x_n)) \to \R((x_1, \cdots, x_n))$ be the following operator:
\[
\Poly[\sum_{\bm{\alpha} \in \Z^n} p_{\bm{\alpha}}\bm{x^{\alpha}}] = \sum_{\bm{\alpha} \in \Z_{\ge 0}^n} p_{\bm{\alpha}}\bm{x^{\alpha}}.
\]

$\Poly[p]$ is not necessarily a polynomial, however, it will be if $p$ is a $d$-homogeneous Laurent series, since there are only finitely many $\bm{\alpha} \in \Z_{\ge 0}^n$ satisfying $|\bm{\alpha}|_1 = d$.  
\begin{definition} [Denormalized Lorentzian Laurent series]
    For a homogeneous Laurent series $p \in \R_{\ge 0}((x_1,\cdots, x_n))$, say that $p$ is a \textit{denormalized Lorentzian (DL) Laurent series} if, for any $\bm{\alpha} \in \Z_{\ge 0}^n$, the polynomial $\Poly[\bm{x}^{\bm{\alpha}}p(\bm{x})]$ is denormalized Lorentzian. 
\end{definition}

\begin{remark}
    The intersection of the class of DL Laurent series and polynomials is the class of DL polynomials. 
    So there is no ambiguity when we say that a polynomial $p$ is DL, without specifying whether it is a DL Laurent series or a DL polynomial.
\end{remark}

\begin{lemma}
\label{bivaraite Laurent series is DL iff it has a log concave sequence}
    Let $p(x,y) = \sum_{n \in \Z} p_nx^{d-n}y^{n}$. Then $p$ is DL if and only if the sequence $\{ p_n\}_{n \in \Z}$ is {log-concave}, that is, it has no internal zeroes and for all $n \in \Z$, we have $p_n^2 \ge p_{n-1}.p_{n+1}$.
\end{lemma}
\begin{proof}
If $p$ is a DL Laurent series, then the polynomial \[ \Poly [x^my^{m} p(x,y) ] = \sum_{-m \le k \le d + m} p_k x^{d-k+m}y^{k+m}\]
    is DL for all $m \in \Z$, and therefore the sequence $\{ p_k \}_{k = -m}^{d+m}$ is log-concave by  \Cref{bivariate homog polynomial is DL iff it has LC coeffs}. Now given $n$, we just need to take a large enough $m$ to see that $p_n^2 \ge p_{n-1}p_{n+1}$. Furthermore, $\{ p_n \}_{n \in \Z}$ has no internal zeros since $\{ p_k \}_{k = -m}^{d+m}$ does not have any internal zeros for all $m \in \Z$.

    For the other direction of the proof, suppose $\{p_n\}_{n \in \Z}$ is a log concave sequence. We need to prove that the polynomial 
    \[
    \Poly[x^{\alpha}y^{\beta}p(x,y)] = \sum_{d + \alpha \ge k \ge - \beta} p_k x^{d - k + \alpha} y^{k + \beta} 
    \]
    is DL. The sequence $\{   p_k\}_{k = -\beta}^{ d+ \alpha}$ inherits log-concavity from $\{ p_n \} _{n \in \Z}$, and $\Poly[x^{\alpha}y^{\beta}p(x,y)]$ is DL by \Cref{bivariate homog polynomial is DL iff it has LC coeffs}.
\end{proof}
DL Laurent series maintain an intimate connection with M-convex sets. 

\begin{lemma}
\label{support of DL Laurent series is M convex}
    The support of any DL Laurent series $p \in \R_{\ge 0}((x_1, \cdots, x_n))$ is M-convex. 
\end{lemma}
\begin{proof}
    Let $\bm{\mu},\bm{\nu} \in \supp(p)$, and let $\mu_i > \nu_i$ for some $i \in [n]$. Take a vector $\bm{\gamma} \in \N^n$ with large enough coordinates $\gamma_1, \cdots, \gamma_n$ so that $\bm{\mu} + \bm{\gamma}, \bm{\nu} + \bm{\gamma} \in \Z_{> 0}^n$. 

    Note that if a vector $\bm{\alpha}$ satisfies $\bm{\alpha} + \bm{\gamma} \ge \bm{0}$, $\bm{\alpha} \in \supp(p)$ if and only if $\bm{\alpha} + \bm{\gamma}$ is in the support of $N[\Poly[\bm{x^{\gamma}}p]]$ because: 
    \[
    [\bm{x}^{\bm{\alpha} + \bm{\gamma}}] \left( N[\Poly[\bm{x^{\gamma}}p]] \right) = [\bm{x}^{\bm{\alpha} + \bm{\gamma}}]\left( \sum_{\bm{d} + \bm{\gamma} \ge \bm{0}} \frac{p_{\bm{d}}}{(\bm{d} + \bm{\gamma})!} \bm{x}^{\bm{d} + \bm{\gamma}} \right) = \frac{[\bm{x^{\alpha}}]p(\bm{x})}{(\bm{\alpha} + \bm{\gamma})!}.
    \]

    By definition, $N[\Poly[\bm{x^{\gamma}}p]]$ is a Lorentzian polynomial, and its support is M-convex by \Cref{M-convexity and Lorentzian}. Note that both $\bm{\mu} + \bm{\gamma}$ and $ \bm{\nu} + \bm{\gamma}$ are elements of the support of $N[\Poly[\bm{x^{\gamma}}p]]$, and we have $(\bm{\mu} + \bm{\gamma})_i = \mu_i + \gamma_i > (\bm{\nu} + \bm{\gamma})_i = \nu_i + \gamma_i$. Therefore, there should exist some $j \in [n]-\{ i \} $ such that $(\bm{\mu} + \bm{\gamma})_j < (\bm{\nu} +\bm{\gamma})_j $ and $\bm{\mu} + \bm{\gamma} - \bm{e}_i + \bm{e}_j, \bm{\nu} + \bm{\gamma} + \bm{e}_i - \bm{e}_j$ are both elements of $\supp(N[\Poly[\bm{x^{\gamma}}p]])$. So $\mu_j < \nu_j$ and $\bm{\mu} - \bm{e}_i + \bm{e}_j, \bm{\nu} + \bm{e}_i - \bm{e}_j \in \supp(p)$ by what we discussed above, and $\supp(p)$ is M-convex by definition.  
\end{proof}

The following is an analog of \Cref{multiple of DL polynomials}.
\begin{lemma}
\label{multiple of a DL Laurent series is DL} 
    Let $p \in \R_{\ge 0}((x_1, \cdots, x_n))$ be a DL Laurent series. Then $x_1^tp(\bm{x})$ is DL for any $t \in \Z$.  
\end{lemma}
\begin{proof}
   It suffices to prove that for any vector $\bm{\alpha} = (\alpha_1, \cdots, \alpha_n) \in \Z_{\ge 0}^n$, the polynomial $\Poly[x_1^t\bm{x^{\alpha}} p(x)]$ is DL. 
    
    Let $M > \alpha_1 + t$ be an integer, and let $\bm{\beta} := (M , \alpha_2 , \cdots, \alpha_n)$. Then $\Poly[\bm{x^{\beta}}p(\bm{x})]$ is a DL polynomial by definition, and \Cref{truncation below preserves DL polynomials} implies that $\trunc_{x_1}^{-(\alpha_1 + t)} [\Poly[\bm{x^{\beta}}p(\bm{x})]]$ is also a DL polynomial. Now observe that:
    \[
        \Poly[\bm{x^{\beta}}p(\bm{x})] = \sum_{\substack{d_1 \ge -M \\ d_i \ge - \alpha_i \hspace{0.1cm} \forall i \ge 2}} p_{\bm{d}} x_1^{d_1 + M}x_2^{d_2 + \alpha_2} \cdots x_n^{d_n + \alpha_n},
    \]   
    and: 
    \begin{align*}
        \trunc_{x_1}^{-(\alpha_1 + t)}\Poly[\bm{x^{\beta}}p(\bm{x})] &= \sum_{\substack{d_1 \ge -M \\ d_1 \ge -(\alpha_1 + t)\\ d_i \ge - \alpha_i \hspace{0.1cm} \forall i \ge 2}} p_{\bm{d}} x_1^{d_1 + M + \alpha_1+ t }x_2^{d_2 + \alpha_2} \cdots x_n^{d_n + \alpha_n}\\
        &= x_1^M \sum_{\substack{d_1 \ge -(\alpha_1 + t)\\ d_i \ge - \alpha_i \hspace{0.1cm} \forall i \ge 2}} p_{\bm{d}} x_1^{d_1 +\alpha_1+ t }x_2^{d_2 + \alpha_2} \cdots x_n^{d_n + \alpha_n}\\
        &= x_1^M \Poly[x_1^t \bm{x^{\alpha}} p(\bm{x})].
    \end{align*}
    We conclude that $x_1^M\Poly[x_1^t \bm{x^{\alpha}} p(\bm{x})]$ is DL, and by \Cref{multiple of DL polynomials}, so is $\Poly[x_1^t \bm{x^{\alpha}} p(\bm{x})]$.
\end{proof}

\begin{corollary}
    \label{equivalent definition for DL Laurent series}
    A homogeneous Laurent series $p \in \R_{\ge 0}((x_1, \cdots, x_n))$ is DL if and only if $\Poly[\bm{x^{\alpha}}p(\bm{x})]$ is DL for any $\bm{\alpha} \in \Z^n$.
\end{corollary}
\begin{proof}
    The backward direction of the proof is trivial. 
    To prove the forward direction, assume that $p$ is a DL Laurent series, and let $\bm{\alpha} \in \Z^n$. By \Cref{multiple of a DL Laurent series is DL}, $\bm{x^{\alpha}}p(\bm{x})$ is a DL Laurent series, and therefore $\Poly[ \bm{x^{\alpha}}p(\bm{x})]$ is a DL polynomial. 
\end{proof}

We will be using \Cref{equivalent definition for DL Laurent series} as an alternate definition for DL Laurent series. The rest of this section is dedicated to exploring some operators that preserve the class of denormalized Lorentzian Laurent series. 

\begin{lemma}
    \label{DL Laurent series are closed}
Let $\{p_i\}_{i \in \N}$ be a sequence of homogeneous DL Laurent series in variables $x_1, \cdots, x_n$, converging coefficient-wise to a Laurent series $p$. Then $p$ is a DL Laurent series itself.
\end{lemma}
\begin{proof}
    It suffices to show that $\Poly[\bm{x^{\alpha}}p(\bm{x})]$ is a DL polynomial for any $\bm{\alpha} \in \Z^n$. Note that the sequence $\{ \Poly[\bm{x^{\alpha}}p_i(\bm{x})]\}_{i \in \N}$ also converges coefficient-wise to $\Poly[\bm{x^{\alpha}}p(\bm{x})]$. So $\Poly[\bm{x^{\alpha}}p(\bm{x})]$ is the limit point of a sequence of DL polynomials, and is DL itself by \Cref{DL polynomials are closed}.
\end{proof}

\begin{lemma}
\label{scaling the variables}
    If $p \in \R_{\ge 0}((x_1,\cdots, x_n))$ is a DL Laurent series. Then, for any $y_1,\cdots, y_n > 0$, the Laurent series $p(y_1x_1, \cdots, y_nx_n)$ is also DL. 
\end{lemma}
\begin{proof}
    Let $\bm{\alpha} \in \Z^n$. Rewrite $\Poly[\bm{x^{\alpha}}p(y_1x_1, \cdots, y_nx_n)]$ as: 
    \begin{align*}
        \Poly[\bm{x^{\alpha}} p(y_1x_1, \cdots, y_nx_n)] &= \sum_{\substack{\bm{d} + \bm{\alpha} \ge \bm{0}\\ \bm{d} \in \Z^n}} p_{\bm{d}} \bm{x}^{\bm{d} + \bm{\alpha}}\bm{y^d}\\ 
        &=\bm{y^{-\alpha}}\Poly[\bm{x^{\alpha}} p](y_1x_1, \cdots, y_nx_n).
    \end{align*}
    
    The statement follows from \Cref{DL polynomials are closed under scaling variables}.
\end{proof}
\begin{lemma}
\label{aux for the next corollary}
    Let $p \in \R_{\ge 0}((x_1, \cdots, x_n))$ be a DL Laurent series. For any $d \in \Z$, the Laurent series $[x_1^d]p \in \R((x_2, \cdots, x_n))$ defined as:
        \[
        [x_1^d]p = \sum_{d_2, \cdots, d_n \in \Z}p_{d, d_2, \cdots, d_n}x_2^{d_2} \cdots x_n^{d_n}
        \]
        is DL as well. 
    
\end{lemma}
\begin{proof}
    Let $q(x_2, \cdots, x_n) =[x_1^{d}] p$. We need to prove that $\Poly[x_{2}^{\alpha_2}\cdots x_n^{\alpha_n}q]$ is DL for any choice of $\alpha_2, \cdots, \alpha_n \in \Z$. Observe that:
    \begin{align*}
        \Poly[x_{2}^{\alpha_2}\cdots x_n^{\alpha_n}q] &=\sum_{d_i+ \alpha_i\ge 0, \forall i \ge 2}p_{(d,d_2, \cdots, d_n)}x_2^{d_2 + \alpha_2} \cdots x_n^{d_n + \alpha_n} \\ &= \Poly[x_1^{-d}x_2^{\alpha_2}\cdots x_n^{\alpha_n}p] (0, x_2, \cdots, x_n)
    \end{align*}
    We know that $\Poly[x_1^{-d}x_2^{\alpha_2}\cdots x_n^{\alpha_n}p] (x_1, x_2, \cdots, x_n)$ is a DL polynomial, and for any $c > 0$, the polynomial $\Poly[x_1^{-d}x_2^{\alpha_2}\cdots x_n^{\alpha_n}p] (c x_1, x_2, \cdots, x_n)$ is still DL by \Cref{DL polynomials are closed under scaling variables}. Note that the class of DL polynomials is closed by \Cref{DL polynomials are closed}, so $\Poly[x_1^{d}x_2^{\alpha_2}\cdots x_n^{\alpha_n}p] (0, x_2, \cdots, x_n)$ as the limit point of $ \{ \Poly[x_1^{-d}x_2^{\alpha_2}\cdots x_n^{\alpha_n}p] (\frac{1}{k} x_1, x_2, \cdots, x_n)\}_{k \in \N}$ is DL and we are done. 
\end{proof}
Applying \Cref{aux for the next corollary} repeatedly to a Laurent series $p$, and using the fact that 
$
    [x_i^{d_i} x_j^{d_j}] p = [x_i^{d_i}] \left( [x_j^{d_j}]p\right)
    $
give us the following corollary: 

\begin{corollary}
\label{the coefficient of x^d in a DL Laurent series is DL}
    Let $p \in \R_{\ge 0}((x_1, \cdots, x_n))$ be a DL Laurent series. 
    Given a subset $I \subseteq [n]$, and $d_i \in \Z$ for all $i \in I$, we have that $[\prod_{i \in I} x_i^{d_i}] p$ is a DL Laurent series in variables $x_j$ for $j \in [n] \setminus I$, where:
    \[
    [\prod_{i \in I}x_i^{d_i}] p = \sum_{d_j \in \Z: \forall j \in [n] \setminus I} p_{\bm{d}} \prod_{j \in [n] \setminus I }x_j^{d_j}.
    \]
\end{corollary}
\begin{corollary}
\label{truncation in Laurent series}
    Generalize $\trunc_{x_1}^t , \ttrunc_{x_1}^t$ from \Cref{truncation below preserves DL polynomials} and \Cref{truncation above preserves DL polynomials} to operators on Laurent series as follows: 
    \[
    \trunc_{x_1}^t \left[\sum_{\bm{\alpha} \in \Z^n } p_{\bm{\alpha}}\bm{x^{\alpha}}\right] = x_i^{-t}\sum_{\substack{\bm{\alpha} \in \Z^n \\ \alpha_1 \ge t }} p_{\bm{\alpha}} \bm{x^{\alpha}},
    \]
    and
    \[
    \ttrunc_{x_1}^t \left[\sum_{\bm{\alpha} \in \Z^n } p_{\bm{\alpha}}\bm{x^{\alpha}}\right] = \sum_{\substack{\bm{\alpha} \in \Z^n \\ \alpha_1 \le t }} p_{\bm{\alpha}} \bm{x^{\alpha}}.
    \]
    These new operators both preserve the class of DL Laurent series, that is, $\trunc_{x_1}^t[p]$ and $\ttrunc_{x_1}^t[p]$ are both DL whenever $p$ is DL.
\end{corollary}
\begin{proof}
    Given that $p(\bm{x})$ is DL, we need to prove that $\Poly[\bm{x^{\beta}}\trunc_{x_i}^t[p]]$ and $\Poly[\bm{x^{\beta}}\ttrunc_{x_i}^t[p]]$ are DL polynomials for any $\bm{\beta} \in \Z^n$. Let $\bm{\gamma} = (-\max\{-\beta_1, t \}, \beta_2, \cdots, \beta_n)$, then:
    \[\Poly[\bm{x^{\beta}}\trunc_{x_i}^t[p]] = x_i^{-t}\sum_{\substack{\alpha_i + \beta_i \ge 0 \forall i \\ \alpha_1 - t \ge 0}} p_{\bm{\alpha}}\bm{x^{\alpha}} = x_i^{-t} \Poly[ \bm{x^{\gamma}} p(\bm{x})],
    \]
    and:
    \[\Poly[\bm{x^{\beta}}\ttrunc_{x_1}^t[p]] = \sum_{\substack{\alpha_i + \beta_i \ge 0 \forall i \\ \alpha_1 \le t}} p_{\bm{\alpha}} \bm{x^{\alpha}} = \ttrunc_{x_1}^t[\Poly[\bm{x^{\beta}}p(\bm{x})]]. 
    \]
    $x_i^{-t} \Poly[\bm{x^{\gamma}}p(\bm{x})]$ is DL by \Cref{multiple of DL polynomials} and \Cref{truncation below preserves DL polynomials}, and $\ttrunc_{x_1}^t[\Poly[\bm{x^{\beta}}p(\bm{x})]]$ is DL by \Cref{truncation above preserves DL polynomials}.
    Therefore both $\Poly[\bm{x^{\beta}}\trunc_{x_i}^t[p]]$ and $\Poly[\bm{x^{\beta}}\ttrunc_{x_i}^t[p]]$ are DL polynomials, and we are done. 
\end{proof}
\subsection{Product of DL Laurent series}
\label{Product of DL Laurent series}
We know from \Cref{multiplication of DL polynomials stays DL} that the class of DL polynomials is closed under multiplication. In \cite{BLP23}, this fact is used to prove that the generating series for contingency tables with bounded entries is DL (see \Cref{subsection: contingency tables} for definition and details), so we naturally search for an analog of \Cref{multiplication of DL polynomials stays DL} for DL Laurent series. But first, we need to address the problem that $\R((x_1, \cdots, x_n))$ itself is not closed under multiplication. Recall that for two Laurent series $p,q \in \R((x_1, \cdots, x_n))$, their product is well defined if the sums $\sum_{\bm{\alpha} + \bm{\beta} = \bm{\gamma}} p_{\bm{\alpha}}q_{\bm{\beta
    }}$ converge in $\R$ for all $\bm{\gamma} \in \Z^n$. Then the product $p\cdot q(\bm{x})$ is defined as:
 \[
    (p\cdot q)(\bm{x}) = \sum_{\bm{\gamma} \in \Z^n} \left( \sum_{\bm{\alpha} + \bm{\beta} = \bm{\gamma}} p_{\bm{\alpha}}q_{\bm{\beta}}\right) \bm{x^{\gamma}}.
    \]
    As a special case, assume that $p,q \in \R((x_1, \cdots, x_n; y_1,\cdots, y_m))$ are two Laurent series in disjoint sets of variables. Say $p$ only uses the $x$ variables and $[\bm{x^{\mu}} \bm{y^{\nu}}]p = 0$ for all $\bm{\nu} \neq \bm{0}$, and $q$ only uses the $y$ variables and $[\bm{x^{\mu}} \bm{y^{\nu}}]q = 0$ for all $\bm{\mu} \neq \bm{0}$. Then $\sum_{\bm{\alpha} + \bm{\beta} = \bm{\gamma}} p_{\bm{\alpha}}q_{\bm{\beta
    }}$ is always a finite sum for any $\bm{\gamma} \in \Z^{n + m}$, and $p.q$ is a well defined Laurent series. It is also easy to see that in this case, if $p,q$ are both DL, their product will be DL: 

\begin{remark}
\label{multiplication of DL Laurent series with disjoint sets of variables}
    If $p(\bm{x}), q(\bm{y})$ are two DL Laurent series in disjoint sets of variables $\bm{x}$ and $\bm{y}$, then $p(\bm{x})q(\bm{y})$ is DL since:
    \begin{align*}
            \Poly[\bm{x^{\alpha}y^{\beta}}p(\bm{x}) q(\bm{y})] &= \sum_{\substack{\bm{\mu} + \bm{\alpha} \ge \bm{0} \\ 
            \bm{\nu} + \bm{\beta} \ge \bm{0}}} p_{\bm{\mu}} q_{\bm{\nu}} \bm{x}^{\bm{\mu} + \bm{\alpha}} \bm{y}^{\bm{\nu}+ \bm{\beta}}
            \\&= \sum_{\bm{\mu} + \bm {\alpha} \ge \bm{0}}p_{\bm{\mu}}\bm{x}^{\bm{\mu} + \bm{\alpha}}. \sum_{\bm{\nu} + \bm {\beta} \ge \bm{0}}q_{\bm{\nu}}\bm{y}^{\bm{\nu} + \bm{\beta}}
            \\
            &= \Poly[\bm{x^{\alpha}}p(\bm{x})] \Poly[\bm{y^{\beta}}q(\bm{y})]
    \end{align*}

    $\Poly[\bm{x^{\alpha}}p(\bm{x})]\Poly[\bm{y^{\beta}}q(\bm{y})]$ is the product of two DL polynomials and is itself DL by \Cref{multiplication of DL polynomials stays DL}. So $\Poly[\bm{x^{\alpha}y^{\beta}}p(\bm{x}) q(\bm{y})]$ is DL for all $\bm{\alpha}, \bm{\beta}$, and $p.q$ is DL by definition. 
\end{remark}

Suppose we have two DL Laurent series $p$ and $q$. Consider two disjoint sets of variables $\bm{x},\bm{y}$, and multiply $p(\bm{x})$ and $q(\bm{y})$ to get a well-defined DL Laurent series $p(\bm{x})q(\bm{y})$. Now in this Laurent series, set $y_1 = x_1$, then $y_2 = x_2$, and so on. We will prove that as long as the final result remains a well-defined Laurent series, it will be DL. 

\begin{definition}[Admissible pairs]
\label{analytical admissibility}
    Let $T_{x_1 = x_2}^{\PS}: \R((x_1, \cdots, x_n)) \to \R^*((x_1, \cdots, x_n))$ be an operator that sends $p(x_1,\cdots, x_n)$ to $p(x_1, x_1, x_3, \cdots, x_n)$. This operator is a generalization of $T^{\Poly}_{x_1 = x_2}$ in \Cref{x1=x2 polynomials}.
    
    We say that $(x_1,x_2)$ is an \textit{admissible pair} of $p$ if $T_{x_1 = x_2}^{\PS}[p]$ is a well defined Laurent series. In other words, if $p$ is given by its coefficients as:
    \[
    p(x_1, \cdots, x_n) = \sum_{\bm{\alpha} \in \Z^n} p_{\bm{\alpha}}\bm{x^{\alpha}},
    \]
    then $(x_1, x_2)$ is an admissible pair of $p$ if for any $\alpha, \alpha_3, \cdots, \alpha_n \in \Z$, the sum $\sum_{\alpha_1+ \alpha_2 = \alpha} p_{(\alpha_1, \alpha_2,\cdots, \alpha_n)}$ converges in $\R$.
    Similarly, one could define $T_{x_i = x_j}^{\PS}$ and admissibility for any pair $(x_i,x_j)$.
\end{definition} 
We will later see in \Cref{admissibility and coefficient of sub Laurent series} that $(x_1,x_2)$ is an admissible pair of $p(x_1, x_2)$ if and only if $(1,1)$ is in the domain of convergence of $p$. 
Now let us prove that the operator $T_{x_1= x_2}^{\PS}$ ``almost" preserves the class of DL Laurent series. 

\begin{theorem}
    \label{analytic admissibility -> x1 = x2 preserves DL}
    Assume that $(x_1, x_2)$ is an admissible pair of a $t$-homogeneous DL Laurent series $p$. 
Then $T_{x_1 = x_2}^{\PS}[p]$ is also a DL Laurent series.

To be more specific, for any given vector $\bm{\alpha} = (\alpha_1, \alpha_3, \cdots, \alpha_n)\in \Z_{\ge 0}^{n-1}$, we will show that:
\[
\Poly\left[x_1^{\alpha_1} x_3^{\alpha_3} \cdots x_n^{\alpha_n}T_{x_1 = x_2}^{\PS}[p] \right]= \lim_{k \to \infty} \Poly[x_1^{\alpha_1-2k}T_{x_1 = x_2}^{\PS}[\Poly[\bm{x^{\beta_k}}p]]],
\]
where $\bm{\beta}_k = (k,k,\alpha_3, \cdots,\alpha_n)$, and the limit is taken in the Euclidean space of homogeneous polynomials of degree at most $|\bm{\alpha}|_1 + t$ in $n-1$ variables.  
\end{theorem}
Note that a similar theorem can be proven for $T_{x_i = x_j}^{\PS}[p]$ when $(x_i,x_j)$ is an admissible pair of $p$.

\begin{proof}
    Let $p(\bm{x}) = \sum_{\bm{d} \in \Z^n} p_{\bm{d}}\bm{x^d}$ be a DL Laurent series and assume that $(x_1, x_2)$ is an admissible pair of $p$. Then :
    \[
    T_{x_1 = x_2}^{\PS}[p]= p(x_1, x_1, x_3, \cdots, x_n) = \sum_{d ,d_3 \cdots, d_n \in \Z} \left(\sum_{d_1 + d_2 = d}p_{d_1, d_2,d_3, 
    \cdots, d_n} \right ) x_1^{d} x_3^{d_3} \cdots x_n^{d_n}.
    \]
    We need to show that for any $\bm{\alpha} = (\alpha_1,\alpha_3, \cdots, \alpha_n) \in \Z_{\ge 0}^{n-1}$, the polynomial $\Poly[\bm{x^{\alpha}} T_{x_1 = x_2}^{\PS}[p]]$ is DL. We have:
    \begin{equation}\label{eq2}
            \Poly[\bm{x^{\alpha}}T_{x_1 = x_2}^{\PS}[p] ] = \sum_{\substack{d + \alpha_1 \ge 0 \\ d_i + \alpha_i \ge 0 \forall i \ge 3}} \left( \sum_{d_1 + d_2= d  } p_{d_1, d_2, d_3, \cdots, d_n}\right) x_1^{d + \alpha_1}x_3^{d_3 + \alpha_3} \cdots x_n^{d_n + \alpha_n}.
    \end{equation}

    So we want to show that the RHS of \Cref{eq2} is a DL polynomial. 
    
     Let $\bm{\beta}_k = (k,k, \alpha_3, \cdots, \alpha_n)$. Since $p$ itself is a DL Laurent series, the polynomial $T_{x_1 = x_2}^{\Poly}[\Poly[\bm{x}^{\bm{\beta}_k}p]]$ is also DL for all $k \in \Z$. Observe that: 
    \begin{align*}
       \Poly[\bm{x}^{\bm{\beta}_k}p]] = \Poly[x_1^{k}x_2^kx_3^{\alpha_3} \cdots x_n^{\alpha_n}p] & =   \sum_{\substack{d_1+k \ge 0 \\ d_2 + k \ge 0\\ d_i + \alpha_i \ge 0 \forall i \ge 3}}p_{\bm{d}}x_1^{d_1 + k} x_2^{d_2 + k } x_3^{d_3 + \alpha_3} \cdots x_n^{d_n + \alpha_n}.
    \end{align*}
   $T_{x_1 = x_2}^{\Poly} [\Poly[\bm{x^{\beta_k}}p]]$ is also a DL polynomial by \Cref{x1=x2 polynomials}:
 \begin{equation} \label{eq3}
        T_{x_1 = x_2}^{\Poly} [\Poly[\bm{x^{\beta_k}}p]]  =   \sum_{\substack{d + 2k \ge 0\\ d_i + \alpha_i \ge 0 \forall i \ge 3}} \left( \sum_{\substack{d_1 + d_2 = d \\ d_1 + k \ge 0 \\d_2 + k \ge 0}}p_{d_1,d_2, \cdots, d_n}\right) x_1^{d + 2k} x_3^{d_3 + \alpha_3} \cdots x_n^{d_n + \alpha_n}.
 \end{equation}
 
 Note that as we take $k $ to $\infty$, the coefficients of \Cref{eq3} converge to their corresponding coefficients in \Cref{eq2}. So let us apply some changes to \Cref{eq3} to get closer to what \Cref{eq2} looks like.
 \begin{align*}
       \Poly[x_1^{\alpha_1 - 2k} T_{x_1 = x_2}^{\Poly} [\Poly[\bm{x^{\beta_k}}p]]] & =   \Poly\left[ \sum_{\substack{d + 2k \ge 0\\ d_i + \alpha_i \ge 0 \forall i \ge 3}} \left( \sum_{\substack{d_1 + d_2 = d \\ d_1 + k \ge 0 \\d_2 + k \ge 0}}p_{d_1,d_2, \cdots, d_n}\right) x_1^{d + \alpha_1} x_3^{d_3 + \alpha_3} \cdots x_n^{d_n + \alpha_n} \right]\\
       & = \sum_{\substack{d + \alpha_1 \ge 0 \\ d + 2k \ge 0\\ d_i + \alpha_i \ge 0 \forall i \ge 3}} \left( \sum_{\substack{d_1 + d_2 = d \\ d_1 + k \ge 0 \\d_2 + k \ge 0}}p_{d_1,d_2, \cdots, d_n}\right) x_1^{d + \alpha_1} x_3^{d_3 + \alpha_3} \cdots x_n^{d_n + \alpha_n}.
 \end{align*}
 For any $k \ge \frac{\alpha_1}{2}$, the first summation would only impose the constraints $d+ \alpha_1 \ge 0$ and $d_i + \alpha_i \ge 0$ for $i \ge 3$. Moreover:
 \[
 \lim_{k \to \infty} \sum_{\substack{d_1 + d_2 = d \\ d_1 + k \ge 0 \\d_2 + k \ge 0}}p_{d_1,d_2, \cdots, d_n} = \sum_{d_1 + d_2 = d}p_{d_1,d_2, d_3, \cdots, d_n},
 \]
 and therefore: 
 \begin{align*}
       \lim_{k \to \infty} \Poly[x_1^{\alpha_1 - 2k} T_{x_1 = x_2}^{\Poly} [\Poly[\bm{x^{\beta_k}}p]]] & = \lim_{k \to \infty} \sum_{\substack{d + \alpha_1 \ge 0 \\ d + 2k \ge 0\\ d_i + \alpha_i \ge 0 \forall i \ge 3}} \left( \sum_{\substack{d_1 + d_2 = d \\ d_1 + k \ge 0 \\d_2 + k \ge 0}}p_{d_1,d_2, \cdots, d_n}\right) x_1^{d + \alpha_1} x_3^{d_3 + \alpha_3} \cdots x_n^{d_n + \alpha_n} \\
       &= \sum_{\substack{d+ \alpha_1 \ge 0 \\ d_i + \alpha_i \ge 0 \forall i \ge 3}}\left( \sum_{d_1 + d_2 = k} p_{d_1, d_2, \cdots, d_n}\right) x_1^{d + \alpha_1}x_3^{d_3 + \alpha_3} \cdots x_n^{d_n + \alpha_n}\\
    &= \Poly[x_1^{\alpha_1} x_3^{\alpha_3} \cdots x_n^{\alpha_n}T_{x_1 = x_2}^{\PS}[p] ].
 \end{align*}
 Observe that LHS of the above equality is DL since the set of DL polynomials is closed by \Cref{DL polynomials are closed}, and therefore $\Poly[x_1^{\alpha_1} x_3^{\alpha_3} \cdots x_n^{\alpha_n}T_{x_1 = x_2}^{\PS}[p] ]$ is DL for any $\alpha_1, \alpha_3, \cdots, \alpha_n \in \Z$.
\end{proof}
For the sake of brevity, we will denote by $T^{\PS}_{x_{i_k} = y_{i_k}, x_{i_{k-1}} = y_{i_{k-1}}, \cdots, x_{i_1} = y_{i_1} }$ the operator $T^{\PS}_{x_{i_k} = y_{i_k}} \circ \cdots \circ T^{\PS}_{x_{i_1} = y_{i_1}} $. The following corollary is obtained by applying \Cref{analytic admissibility -> x1 = x2 preserves DL} repeatedly.
\begin{corollary}
\label{multiplication of DL Laurent series with admissible pairs}
Suppose $p(\bm{x}),q(\bm{x})$ are two homogeneous DL Laurent series. Further assume that $(x_i, y_i)$ is an admissible pair of $T^{\PS}_{x_{i-1}=y_{i-1}, \cdots, x_1 = y_1}[p(\bm{x}) q(\bm{y})]$ for any $i \in [n]$ (i.e., $p(\bm{x}) q(\bm{x})$ is a well defined Laurent series). Then $p(\bm{x})q(\bm{x})$ is also a DL Laurent series.   
\end{corollary}

Roughly speaking, \Cref{multiplication of DL Laurent series with admissible pairs} states that a well-defined product of DL Laurent series remains DL. Throughout the rest of this paper, we will invoke this result without checking the details.

\begin{remark}
    Assuming that $(x_1,y_1)$ and $(x_2,y_2)$ are both admissible pairs of $p(\bm{x}, \bm{y})$, $(x_2,y_2)$ is not necessarily an admissible pair of $T_{x_1 = y_1}^{\PS}[p]$. For instance, $(x_1,y_1)$ and $(x_2,y_2)$ are both admissible pairs of the Laurent series $p(x_1,x_2,y_1,y_2) = \sum_{n \ge 0} x_1^n y_1^{-n} x_2^n y_2^{-n}$, but $(x_2, y_2)$ is not an admissible pair of $T_{x_1 = y_1}^{\PS}[p] = \sum_{n \ge 0} x_2^n y_2^{-n}$. 
    \\ So it is essential in the statement of \Cref{multiplication of DL Laurent series with admissible pairs} to assume that $(x_i,y_i)$ is an admissible pair of $T^{\PS}_{x_{i-1}=y_{i-1}, \cdots, x_1 = y_1}[p(\bm{x})q(\bm{y})]$, and not an admissible pair of $p(\bm{x})q(\bm{y})$.
\end{remark}

\section{Capacity bound for the coefficients of DL Laurent series}
\label{section: capacity for Laurent series}
With a better understanding of basic properties of DL Laurent series, we now aim to derive lower bounds for their coefficients. The problem of bounding coefficients of log-concave polynomials is a well-studied topic, as discussed in the introduction. We will use tools that have been employed before to find such bounds. To apply these methods in our setting, however, we must first analyze the domains of convergence of DL Laurent series.

We define and study these domains in \Cref{domain of convergence}. The main result of that section, \Cref{omega pk and how it is related to omega p}, will be used in \Cref{capacity bounds}, alongside other tools such as capacity, to obtain a lower bound on the coefficients of DL Laurent series. This bound is stated in \Cref{a lower bound for coefficients using capacity}.

\subsection{Domains of convergence}
\label{domain of convergence}
    The goal of this section is to study the domain of convergence of DL Laurent series. We will later use domains of convergence to generalize the definition of Gurvits's capacity function for Laurent series in \Cref{capacity bounds}. The capacity function will then be used to prove lower bounds for the coefficients of DL Laurent series. 
  \begin{definition} [Domain of convergence]
    \label{domain of convergence definition}
    Let $p \in \R_{\ge 0}((x_1, \cdots, x_n))$. For some $\bm{x} \in \R_{> 0}^n$, say that $p(\bm{x})$ converges absolutely to $L \in \R$ if the following limit converges to $L$:
    \[ \lim_{k \to \infty} \sum_{\forall i: \hspace{0.1cm}|\alpha_i| \le k }p_{\bm{\alpha}}\bm{x^{\alpha}}.\] 
    Then the \textit{domain of convergence} of $p$, denoted by $\Omega_p \subseteq \R_{> 0}^n$, is the set of all points $\bm{x} \in \R_{>0}^n$ for which $p(\bm{x})$ converges. If $p\equiv 0$ identically, let $\Omega_p = \R_{>0}^n$.
    We are purposefully excluding points $\bm{x} \in \R_{\ge 0}^n \setminus \R_{>0}^n$ because our definition of Laurent series allows negative powers, and working with such points brings up unnecessary complications. 
    \end{definition}

    For a DL Laurent series $p \in \R((x_1,\cdots, x_n))$, we seek a recursive method to obtain $\Omega_p$
  from domains of convergence of Laurent series in fewer than $n$ variables. Such a recursion will be useful for the inductive arguments in \Cref{capacity bounds}. 
  
  We will begin by stating the main result of this section (\Cref{omega pk and how it is related to omega p}), which relates $\Omega_p$ to $\Omega_{[x_n^k] p}$ for $k \in \Z$. Then we will explore \Cref{omega pk and how it is related to omega p}'s consequences in \Cref{if degree set is bounded then domain of convergence can be lifted} through \Cref{projection of p's domain of convergence}. The groundwork for the proof of \Cref{omega pk and how it is related to omega p} is developed in \Cref{constant multiple of a point in the domain of convergence} through \Cref{moving the boundary case in to the interior}, and the proof of \Cref{omega pk and how it is related to omega p} is given at the end of this section. 
    
   \begin{proposition}
    \label{omega pk and how it is related to omega p}
    Assume that $p$ is a $d$-homogeneous DL Laurent series in $x_1, \cdots,x_n$. Further assume that: 
    \[
    p(x_1, \cdots, x_n) = \sum_{k \in \Z} x_n^k p_k(x_1, \cdots, x_{n-1}).
    \]
    The following statements hold for $p$:
    \begin{enumerate}
        \item Let $(y_1, \cdots, y_{n-1}) \in \bigcap_{k}\Omega_{p_k}$. Then $\{p_k(y_1, \cdots, y_{n-1}) \}_{k}$ is a log-concave sequence with no internal zeros,
        \item $\Omega_{p_m} = \Omega_{p_k}$ whenever $p_m, p_k$ are both non-zero Laurent series, 
        \item Let $(z_1, \cdots, z_{n-1}) \in \Omega_{p_k}$ for some nonzero $p_k$.
        Let $y$ be such that
        \[
        y \in \left( \inf_{m \in \Z} \frac{p_{m}(z_1, \cdots, z_{n-1})}{p_{m+1}(z_1, \cdots, z_{n-1})} , \sup_{m \in \Z}\frac{p_m(z_1,\cdots, z_{n-1})}{p_{m+1}(z_1, \cdots, z_{n-1})}
        \right).\]
        Then $(z_1, \cdots, z_{n-1}, y) \in \Omega_p$. 
    \end{enumerate}
\end{proposition}

\begin{remark}
\label{How to handle 0/0 in inf and sup}
    Recall that if $\{p_m\}_{m \in \Z}$ is a log-concave sequence, then we have
    \[
        \cdots \leq \frac{p_{m-1}}{p_m} \leq \frac{p_m}{p_{m+1}} \leq \frac{p_{m+1}}{p_{m+2}} \leq \cdots.
    \]
    So if $p_m > 0$ for all $m$, the interval given above in \Cref{omega pk and how it is related to omega p} (3) can be written as
    \[
        \left( \lim_{m \to -\infty} \frac{p_{m}(z_1, \cdots, z_{n-1})}{p_{m+1}(z_1, \cdots, z_{n-1})} , \lim_{m \to +\infty}\frac{p_m(z_1,\cdots, z_{n-1})}{p_{m+1}(z_1, \cdots, z_{n-1})}
        \right).
    \]
    When the sequence $\{p_m\}_{m \in \Z}$ is eventually zero in either direction however, more care must be taken to understand how this interval is actually defined. Concretely, if $p_m = 0$ for all $m \geq M$ then we adopt the convention
    \[
        \sup_{m \in \Z}\frac{p_m(z_1,\cdots, z_{n-1})}{p_{m+1}(z_1, \cdots, z_{n-1})} = \lim_{m \to +\infty}\frac{p_m(z_1,\cdots, z_{n-1})}{p_{m+1}(z_1, \cdots, z_{n-1})} = +\infty,
    \]
    and if $p_m = 0$ for all $m \leq M$ then we adopt the convention
    \[
        \inf_{m \in \Z}\frac{p_m(z_1,\cdots, z_{n-1})}{p_{m+1}(z_1, \cdots, z_{n-1})} = \lim_{m \to -\infty}\frac{p_m(z_1,\cdots, z_{n-1})}{p_{m+1}(z_1, \cdots, z_{n-1})} = 0.
    \]
    With these conventions, \Cref{omega pk and how it is related to omega p} (3) holds and the interval given is as large as possible (within $\R_{>0}$).

\end{remark}

It follows from the third statement of \Cref{omega pk and how it is related to omega p} 
that if the sequence $\{p_j(z_1,\cdots, z_{n-1})\}_{j \in \Z}$ is not geometric, then $(z_1, \cdots, z_{n-1})$ can be extended to a point $(z_1, \cdots, z_{n-1},y)$ in $\Omega_p$. In particular, if $p_m \equiv 0$ for some index $m$ (or in other words if the set of powers of $x_n$ in $p$ is bounded below or above), then $\{p_j(z_1,\cdots, z_{n-1})\}_{j \in \Z}$ is not a geometric sequence. Observe that any point $(z_1, \cdots, z_{n-1}, y) \in \Omega_p$ can be traced back to a point $(z_1, \cdots, z_{n-1}) \in \Omega_{p_k}$ as well.

\begin{definition}[{The set $\Deg_{x_i}^{\bm{\alpha}}$}]
\label{degree set}
    Let $p(x_1,\cdots, x_n)$ be a Laurent series and $\bm{\alpha} \in \Z^n$ be a given vector. For any $2 \le i \le n $, let $\Deg_{x_i}^{\bm{\alpha}}(p)$ denote the set $$\{k_i : (k_1, \cdots, k_i , \alpha_{i+1}, \cdots, \alpha_n) \in \supp(p) \text{ for some } k_1, \cdots, k_{i-1} \in \Z \},$$ and let $\Deg_{x_i}(p)$ denote the set \[\{ k_i : (k_1, \cdots, k_i, \cdots, k_n) \in \supp(p) \text{ for some } k_1, \cdots, k_{i-1}, k_{i+1} , \cdots, k_n \in \Z \}.\] 
\end{definition}
The following remarks and corollary are direct consequences of \Cref{omega pk and how it is related to omega p}. 
\begin{remark}
\label{if degree set is bounded then domain of convergence can be lifted}
    Suppose $\Deg_{x_n}(p)$ is either bounded below or above. For any $\bm{z} = (z_1, \cdots, z_{n-1}) \in \Omega_{p_k}$, the following interval:
    \[
    \left(\inf_{m \in \Z} \frac{p_{m}(z_1, \cdots, z_{n-1})}{p_{m+1}(z_1, \cdots, z_{n-1})} , \sup_{m \in \Z}\frac{p_m(z_1,\cdots, z_{n-1})}{p_{m+1}(z_1, \cdots, z_{n-1})}
        \right)
    \]
    is nonempty. Therefore, by the third statement of \Cref{omega pk and how it is related to omega p}, $\bm{z}$ can always be extended to some $(z_1, \cdots, z_{n-1},y) \in \Omega_p$. 
\end{remark}

\begin{remark}
\label{any point in the domain of convergence of a p gives a log concave sequence pk}
    If $(y_1, \cdots, y_n) \in \Omega_p$ then $(y_1, \cdots, y_{n-1}) \in \Omega_{p_k}$ for any $k \in \Z$ since: 
    \[
    y_n^k p_k(y_1, \cdots, y_{n-1}) \le p (y_1, \cdots, y_n) < \infty
    \]
    and therefore $p_k(y_1, \cdots, y_n)$ converges as well. 
\end{remark}

\begin{corollary}
\label{projection of p's domain of convergence}
    If $\Deg_{x_n}(p)$ is bounded below or above, then by \Cref{if degree set is bounded then domain of convergence can be lifted} and \Cref{any point in the domain of convergence of a p gives a log concave sequence pk}:
    \[
    \pi_n(\Omega_p) = \Omega_{p_k},
    \]
    where $\pi_n: \R^n \to \R^{n-1}$ is defined as $\pi_n(x_1, \cdots, x_n) = (x_1, \cdots, x_{n-1})$.
\end{corollary}

We will now tend to the proof of \Cref{omega pk and how it is related to omega p}.

\begin{remark}
    \label{constant multiple of a point in the domain of convergence}
    Assume that $p \in \R_{\ge 0} ((x_1, \cdots, x_n))$ is $d$-homogeneous. If $\bm{x} \in \Omega_p$, then $\lambda \bm{x} \in \Omega_p$ for any $\lambda > 0$ since $p(\lambda \bm{x}) = \lambda^d p(\bm{x})$. 
\end{remark}

\begin{lemma}
\label{admissibility and coefficient of sub Laurent series}
Let $p \in \R_{\ge 0} ((x_1, \cdots, x_n))$ be a $d$-homogeneous Laurent series. Then $(x_1, x_2)$ is an admissible pair of $p$ if and only if $(1,1) \in \Omega_{[x_3^{\alpha_3} \cdots x_n^{\alpha_n}]p}$ for any $\alpha_3, \cdots, \alpha_n \in \Z$.   
   
\end{lemma}
\begin{proof}
     Write: $$[x_3^{\alpha_3} \cdots x_n^{\alpha_n}]p = \sum_{\alpha_1+ \alpha_2 = d - (\alpha_3 + \cdots + \alpha_n)} p_{\bm{\alpha}} x_1^{\alpha_1} x_2^{\alpha_2}. $$
    $(1,1) \in \Omega_{[x_3^{\alpha_3} \cdots x_n^{\alpha_n}]p}$ for all $\alpha_3, \cdots, \alpha_n \in \Z$ if and only if $\sum_{\alpha_1+\alpha_2 = d - (\alpha_3 + \cdots + \alpha_n)} p_{\bm{\alpha}}$ is convergent for all $\alpha_3, \cdots, \alpha_n \in \Z$, which is exactly the condition we need for admissibility of $(x_1, x_2)$ from \Cref{analytical admissibility}.
\end{proof}
 \Cref{moving the boundary case in to the interior} and \Cref{domain of convergence is non-empty for non trivial DL bivariate Laurent series} will be our base cases for the proof of the second and the third statements of \Cref{omega pk and how it is related to omega p} respectively.
    
    \begin{lemma}
\label{moving the boundary case in to the interior}
Let $n \ge 2$ and assume that $p \in \R_{\ge 0}((x_1, \cdots, x_n))$ is a $d$-homogeneous DL Laurent series linear in $x_n$, given by:
\[
p(x_1, \cdots,x_n) = x_n s(x_1, \cdots, x_{n-1}) + r(x_1, \cdots, x_{n-1}).
\]
Then $\Omega_r = \Omega_s$.
\end{lemma}
\begin{proof}
    We just need to prove that some vector $\bm{a} = (a_1, \cdots, a_{n-1})$ is an element of $\Omega_r$ if and only if it is an element of $\Omega_s$. Note that we can further assume $\bm{a} = \bm{1}_{n-1} = (1,\cdots, 1) \in \R^{n-1}$, since we can scale the variables of $p$, and work with $p(x_1/a_1, \cdots, x_n/a_n)$ instead (which is still a DL Laurent series by \Cref{scaling the variables}). So we want to prove that $\bm{1}_{n-1} \in \Omega_r $ if and only if $ \bm{1}_{n-1} \in \Omega_s$. 

    If $n = 2$, $p$ is a homogeneous bivariate Laurent series of degree $d$ and linear in $x_2$, so we should have:
    \[
    p(x_1, x_2) = a x_2 x_1^{d-1} +  b x_1^d
    \]
    for some $a,b \in \R_{\ge 0}$, and the statement is trivial. 
    
    Let $n \ge 3$. The proof is done by induction over $n$. The induction basis for $n =3$ is proven in \Cref{base case for induction proving that domain of convergence of two Laurent series is equal}. There are two facts at the core of the proof for $n =3$. First, there is a correspondence between the elements of $\supp(s)$ and the elements of $\supp(r)$ due to the M-convexity of $\supp(p)$. Second, the growth of the coefficients of $s$ and $r$ is roughly similar because $p$ is log-concave. We note that this second point relies directly on the one positive eigenvalue condition which defines (denormalized) Lorentzian.

    Now let $n > 3$. We would like to set $x_1 = x_2$ to get $T_{x_1 = x_2}^{\text{LS}} [p] = x_n T_{x_1 = x_2}^{\text{LS}}[s] + T_{x_1 = x_2}^{\text{LS}}[r]$. We can then use induction hypothesis for $T^{\text{LS}}_{x_1 = x_2}[p]$, which is a Laurent series in $n-1$ variables. To set $x_1 = x_2$ however, we need to show that $(x_1, x_2)$ is an admissible pair of $p$. To argue that this is the case, we will use the statement of the induction basis for $n = 3$. This is the main reason why this induction starts at $n = 3$, and not $n = 2$.
    
    Let us first show that $(x_1,x_2)$ is an admissible pair of $p$. For any $k_3, \cdots, k_{n-1} \in \Z$, write: 
    \[
    [x_3^{k_3} \cdots x_{n-1}^{k_{n-1}}]p = x_n [x_3^{k_3} \cdots x_{n-1}^{k_{n-1}}]s + [x_3^{k_3} \cdots x_{n-1}^{k_{n-1}}]r
    .\]
    By \Cref{the coefficient of x^d in a DL Laurent series is DL}, all of the Laurent series $s,r, [x_3^{k_3} \cdots x_{n-1}^{k_{n-1}}]p, [x_3^{k_3} \cdots x_{n-1}^{k_{n-1}}]s$ and $[x_3^{k_3} \cdots x_{n-1}^{k_{n-1}}]r$ are DL. Note that $[x_3^{k_3} \cdots x_{n-1}^{k_{n-1}}]p $ is a Laurent series with $3$ variables $x_1,x_2, x_n$ and is linear in $x_n$. So by the induction hypothesis, $(1,1) \in \Omega_{[x_3^{k_3} \cdots x_{n-1}^{k_{n-1}}]s} $ if and only if $ (1,1) \in \Omega_{[x_3^{k_3} \cdots x_{n-1}^{k_{n-1}}]r}$.

    Suppose $\bm{1}_{n-1} \in \Omega_s$, we need to show that $\bm{1}_{n-1} \in \Omega_r$. Since $([x_3^{k_3} \cdots x_{n-1}^{k_{n-1}}]s)(1,1) < s (\bm{1}_{n-1})$, we have that $(1,1) \in \Omega_{[x_3^{k_3} \cdots x_{n-1}^{k_{n-1}}]s}$, and using the induction hypothesis, $ (1,1) \in \Omega_{[x_3^{k_3} \cdots x_{n-1}^{k_{n-1}}]r}$. This holds for any choice of $k_3, \cdots, k_{n-1}$, which by \Cref{admissibility and coefficient of sub Laurent series}, means that $(x_1, x_2)$ is an admissible pair of both $s$ and $r$, implying that $(1,1)$ is an admissible pair of $p$ as well. 
    
    Let $p_1 = T_{x_1 = x_2}^{\PS}[p]$. $p_1$ is DL by \Cref{analytic admissibility -> x1 = x2 preserves DL}, and:
    \[
    p_1(x_1, x_3, \cdots, x_n) = x_n  s_1(x_1, x_3, \cdots, x_{n-1}) + r_1(x_1, x_3, \cdots,x_{n-1}) ,
    \]
    where $s_1 = T_{x_1 = x_2}^{\PS}[s]$ and $r_1 = T_{x_1 = x_2}^{\PS}[r]$. Going back to our initial assumption, we have $\bm{1}_{n-1} \in \Omega_s$, giving us $\bm{1}_{n-2} \in \Omega_{s_1}$. By the induction hypothesis for $s_1$ and $ r_1$, $\bm{1}_{n-2} \in \Omega_{r_1}$, which means that $r_1(\bm{1}_{n-2})$ is convergent, and so is $r(\bm{1}_{n-1}) = r_1(\bm{1}_{n-2})$. Thus, we have proven $\bm{1}_{n-1} \in \Omega_{r}$, and we are done. The other direction follows similarly. 
\end{proof}

 \begin{definition} [Trivially DL]
    \label{DL and coefficient sequence connection, non-trivially DL}
    As previously discussed in \Cref{bivaraite Laurent series is DL iff it has a log concave sequence}, $p(x,y) = \sum_{m \in \Z}p_m x^{d- m} y^{m}$ is a DL Laurent series if and only if $\{p_m\}_{m \in \Z}$ is a log-concave sequence. Further say that $p(x,y)$ is \textit{trivially DL} if $\{p_m\}_{m \in \Z}$ is a geometric progression, that is, if there exists $a,c>0$ such that $p_m = ac^m$ for all $m \in \Z$. 
\end{definition}
 \begin{lemma}
    \label{domain of convergence is non-empty for non trivial DL bivariate Laurent series}
    Let $p(x,y) = \sum_{m \in \Z}p_m x^{d- m} y^{m}$ be a DL Laurent series. $\Omega_p$ is empty if and only if $p$ is trivially DL. Moreover:
    \[ \{ 1 \} \times \left(  \inf_{m \in \Z}\frac{p_{m}}{p_{m+1}}, \sup_{m \in \Z} \frac{p_{m}}{p_{m+1}}\right) \subseteq \Omega_p. \]

\end{lemma}
\begin{proof}
    First observe that $\{p_m\}_{m \in \Z}$ is a log-concave sequence by \Cref{bivaraite Laurent series is DL iff it has a log concave sequence}, so $p_m^2 \ge p_{m-1}p_{m+1}$ for all $m \in \Z$, and assuming that $p_m \neq 0$, $\left(\inf_{m \in \Z} \frac{p_{m}}{p_{m+1}}, \sup_{m \in \Z} \frac{p_m}{p_{m+1}}\right) \subseteq \R_{>0}$ is a (possibly empty) interval.
    
    We will prove that $\{ 1 \} \times \left( \inf_{m \in \Z} \frac{p_{m}}{p_{m+1}},\sup_{m \in \Z} \frac{p_m}{p_{m+1}}\right) \subseteq \Omega_p$ for any DL Laurent series $p(x,y)$. Then for non-trivially DL Laurent series, this immediately implies that the domain of convergence is non-empty since $p_m^2 > p_{m-1}p_{m+1}$ for some $m \in \Z$. we will also have to show that the domain of convergence of any trivially DL Laurent series is empty. 
    
    Fix $z \in \left(\inf_{m \in \Z} \frac{p_{m}}{p_{m+1}}, \sup_{m \in \Z} \frac{p_m}{p_{m+1}}\right)$. Write $p(1,z) = \sum_{n \in \Z} p_n z^n$, and define $q_1(z) := \sum_{n \ge 0} p_{n}z^n$ and $q_2(z):=\sum_{n \ge 0} p_{-n}z^{-n}$. $p(1,z)$ converges if and only if $q_1(z)$ and $q_2(z)$ both converge.

    To show that $q_1(z)$ is convergent, we can reduce to the case where $q_1$ is an infinite sum.  Since $\{p_m\}_{m \in \Z}$ has no internal zeroes, all the coefficients $p_m$ for $m \ge 0$ are non-zero. Log-concavity of the sequence $\{ p_m \}_{m \in \Z}$, implies $\{\frac{p_m}{p_{m+1}}\}_{m \in \Z}$ is an increasing sequence. Therefore, 
    \[
    \lim_{m \to \infty} \frac{p_{m+1}z^{m+1}}{p_{m}z^{m}} = z \cdot \lim_{m \to \infty} \frac{p_{m+1}}{p_m} =  \frac{z}{\lim_{m \to \infty} \frac{p_m}{p_{m+1}}} =  \frac{z}{\sup_{m \in \Z} \frac{p_m}{p_{m+1}}}< 1 \]
     The ratio test from \Cref{ratio test} implies the convergence of $q_1$. 
    
    Similarly, if $q_2(z)$ is a finite power sum, $q_2$ is convergent, so we can assume that $p_{m}$ is nonzero for any $m \ge 0$. Note that $\{\frac{p_{m+1}}{p_m} \}_{m \in \Z}$ is a decreasing sequence. Thus we can compute:   
    \[
    \lim_{m \to \infty} \frac{p_{-(m+1)}z^{-(m+1)}}{p_{-m}z^{-m}} = z^{-1}\cdot \lim_{m \to - \infty}\frac{p_{m+1}}{p_m} = z^{-1} \cdot \inf_{m \in \Z} \frac{p_{m+1}}{p_m} < 1,
    \]
    and $q_2$ is convergent as well by \Cref{ratio test}. So far, we have proven: 
    \[ \{ 1 \} \times \left( \inf_{m \in \Z} \frac{p_{m}}{p_{m+1}}, \sup_{m \in \Z} \frac{p_{m}}{p_{m+1}}\right) \subseteq \Omega_p,\]
    implying $\Omega_p \neq \emptyset$ for non-trivially DL Laurent series. Finally, suppose $\{p_m\}_{m \in \Z}$ is a geometric progression, and $p_m = p_0 a^m$ for some $a \in \R_{> 0}$. Rewrite $p(x,y)$ as:
    \[
        p(x,y) = \sum_{m \in \Z} p_m x^{d-m}y^{m} = p_0 x^d \sum_{m \in \Z} (\frac{ay}{x})^m .
    \]
    It is easy to check that the sum $\sum_{m \in \Z} (\frac{ay}{x})^m$ does not converge for any $x,y \in \R_{>0}$, and $\Omega_p$ is empty. 
\end{proof}

Now we are ready to prove \Cref{omega pk and how it is related to omega p}. 
\begin{proof} [Proof of \Cref{omega pk and how it is related to omega p}]
     To prove the first statement, assume that $(y_1, \cdots, y_{n-1}) \in \Omega_k$ for any $k$. Ideally, if \\
     $T_{x_1 = x_2, x_2 = x_3, =\cdots, x_{n-2} = x_{n-1}}^{\PS}[p](s,t) = p(s, \cdots, s, t)$ was still a well defined DL Laurent series, we could have simply used \Cref{scaling the variables} to prove that the following Laurent series is DL:
     \[
     p(y_1s, \cdots y_{n-1}s,t) = \sum_{k \in \Z} t^k p_k(y_1s, \cdots, y_{n-1}s) =\sum_{k \in \Z} t^k s^{d-k} p_k(y_1, \cdots, y_{n-1}) ,
     \]
    and by \Cref{bivaraite Laurent series is DL iff it has a log concave sequence}, $\{p_k(y_1, \cdots, y_{n-1}) \}_{k \in \Z}$ would have been a log-concave sequence. But we are not assuming anything about the admissibility of $(x_i, x_j)$. Instead, we will utilize the $\Poly$ operator and a similar argument for $\Poly[\bm{x^{\alpha}} p]$. 
     
     First, note that it suffices to prove that for any given $l\le u$, the sequence $\{p_k(y_1, \cdots, y_{n-1}) \}_{k = l}^{u}$ is log-concave. For any $\bm{\alpha}= (\alpha_1, \cdots, \alpha_{n-1}) \in \Z^{n-1} $, we have:
     \[
     \Poly[x_1^{\alpha_1} \cdots x_{n-1}^{\alpha_{n-1}}x_n^{-l} p](x_1,\cdots, x_n) = \sum_{k \ge l} x_n^{k-l} \Poly[x_1^{\alpha_1}\cdots x_{n-1}^{\alpha_{n-1}}p_k](x_1, \cdots, x_{n-1}).
     \]
    For the sake of brevity, define:
    \begin{align*}
        q_{\bm{\alpha}}(x_1,\cdots, x_n)&:= \Poly[x_1^{\alpha_1} \cdots x_{n-1}^{\alpha_{n-1}}x_n^{-l} p](x_1,\cdots, x_n),\\
        q_{k,{\bm{\alpha}}} (x_1, \cdots, x_{n-1})&:= \Poly[x_1^{{\alpha_1}}\cdots x_{n-1}^{\alpha_{n-1}}p_k](x_1, \cdots, x_{n-1}).
    \end{align*}
   So $q_{\bm{\alpha}}(x_1, \cdots, x_n) = \sum_{k \ge l} x_n^{k-l} q_{k, \bm{\alpha}}(x_1, \cdots, x_{n-1})$. $q_{\bm{\alpha}}$ is a DL polynomial by definition, giving us that $q_{\bm{\alpha}}(y_1t, \cdots, y_{n-1}t,s)$ is a $(d + |\bm{\alpha}|_1 - l)$-homogeneous bivariate DL polynomial in $t$ and $s$ by \Cref{DL polynomials are closed under scaling variables} and \Cref{x1=x2 polynomials}. Rewrite $q_{\bm{\alpha}}(y_1t, \cdots, y_{n-1}t, s)$ as:
\begin{align*}
      q_{\bm{\alpha}}(y_1t, \cdots, y_{n-1}t,s) &= \sum_{k \ge l} s^{k-l} t^{d + |\bm{\alpha}|_1 - k} q_{k , \bm{\alpha}}(y_1, \cdots, y_{n-1}).
\end{align*}
    We conclude that the sequence of coefficients of $q_{\bm{\alpha}}(y_1t, \cdots, y_{n-1}t, s)$,  $\{ q_{k,{\bm{\alpha}}}(y_1, \cdots, y_{n-1})\}_{k \ge l}$ is log-concave by \Cref{bivariate homog polynomial is DL iff it has LC coeffs} (note that this sequence is log-concave regardless of the fact that $(y_1, \cdots, y_{n-1}) \in \cap_{k}\Omega_{p_k}$. We will use this in the proof of the second statement of this theorem).
    
    Assume further that $(y_1, \cdots, y_{n-1}) \in \Omega_{{p_k}}$ for all $k \in \Z$. Using \Cref{bivariate homog polynomial is DL iff it has LC coeffs} again, for any $u \ge l$, the polynomial
    $\sum_{k = l}^u s^{k-l}t^{u-k}q_{k,\bm{\alpha}}(y_1, \cdots, y_{n-1})$ is DL, and since the class of DL polynomials is closed by \Cref{DL polynomials are closed}, the polynomial $\lim_{\bm{\alpha}\to \infty}\sum_{k = l}^u s^{k-l}t^{u-k}q_{k,\bm{\alpha}}(y_1, \cdots, y_{n-1})$ is also DL. We have:
    \begin{align*}
        \lim_{\bm{\alpha}\to \infty}\sum_{k = l}^u s^{k-l}t^{u-k}q_{k,\bm{\alpha}}(y_1, \cdots, y_{n-1}) &= \sum_{k = l}^u s^{k-l}t^{u-k} \lim_{\bm{\alpha} \to \infty}q_{k,\bm{\alpha}}(y_1, \cdots, y_{n-1})\\
        &= \sum_{k = l}^u s^{k-l}t^{u-k} p_k(y_1, \cdots, y_{n-1}).
    \end{align*}
     So the coefficient sequence of this polynomial, $\{ p_k(y_1, \cdots, y_{n-1})\}_{k =l} ^u $, is log concave, and we are done.

     Let us move on to the proof of the second statement. We have proven in the first part of this theorem that for any arbitrary $\bm{\alpha} \in \Z^n$, $(y_1, \cdots, y_{n-1}) \in \R_{>0}^{n-1}$ and $l\le u$, the sequence $\{q_{k,{\bm{\alpha}}} \}_{k = l}^u$ is log-concave, and has no internal zeroes. $q_{k, \bm{\alpha}}$ is a ``part" of the Laurent series $p_k$, and $\{p_k\}_{k \in \Z}$ cannot have any internal zeros either, in the sense that if $p_k \equiv 0$, then $p_m \equiv 0$ either for all $m \ge k$ or for all $m \le k$. So now, it suffices to show that $\Omega_{p_{k}} = \Omega_{p_{k+1}}$ whenever $p_k,p_{k+1}$ are non-zero. Take the Laurent series: 
    \[
    \trunc_{x_n}^{k}\circ\ttrunc_{x_n}^{k+1}[p] = p_k(x_1, \cdots, x_{n-1}) + x_n p_{k+1}(x_1, \cdots, x_{n-1}).
    \]
    $\trunc_{x_n}^{k}\circ\ttrunc_{x_n}^{k+1}[p]$ is DL by \Cref{truncation in Laurent series}, and $\Omega_{p_k} = \Omega_{p_{k+1}}$ by \Cref{moving the boundary case in to the interior}. We are done. 
    
    To prove the third statement, take some $(z_1, \cdots, z_{n-1}) \in \Omega_{p_k}$. We have that $(z_1, \cdots, z_{n-1}) \in \Omega_{p_j}$ for any nonzero $p_j$ by statement 2, and therefore by statement 1, $\{ p_j(z_1, \cdots, z_{n-1})\}_{j \in \Z}$ is a log-concave sequence, and the Laurent series
    \[
    q(s,t) = \sum_{j \in \Z} p_j(z_1, \cdots, z_{n-1}) s^j t^{-j}
    \]
    is a DL bivariate Laurent series. Apply \Cref{domain of convergence is non-empty for non trivial DL bivariate Laurent series} to get that $q(s,t)$ converges for any
    $\frac{t}{s}$ in interval $\left(  \inf_{m \in \Z} \frac{p_{m}(z_1, \cdots, z_{n-1})}{p_{m+1}(z_1, \cdots, z_{n-1})},\sup_{m \in \Z} \frac{p_m(z_1,\cdots, z_{n-1})}{p_{m+1}(z_1, \cdots, z_{n-1})}\right)$. Equivalently, for any $y$ such that:
     \[
        y \in \left( \inf_{m \in \Z} \frac{p_{m}(z_1, \cdots, z_{n-1})}{p_{m+1}(z_1, \cdots, z_{n-1})}, \sup_{m \in \Z} \frac{p_m(z_1,\cdots, z_{n-1})}{p_{m+1}(z_1, \cdots, z_{n-1})}
        \right),\]
    we have that the Laurent series $q(y,1)=p(z_1,\cdots,z_{n-1},y)$ converges. 
\end{proof}

\subsection{Capacity bounds}
\label{capacity bounds}
We will define and study the capacity function for Laurent series. Then using results from \Cref{domain of convergence}, we will prove a lower bound similar to \Cref{theorem from BLP} for the coefficients of DL Laurent series. 

\begin{definition}
    Recall that the capacity of a polynomial $p \in \R_{\ge 0}[x_1, \cdots, x_n]$ at the point $\bm{\alpha} \in \Z^n$ is the function $ \cpc_{\bm{\alpha}}(p) = \inf \frac{p(\bm{x})}{\bm{x^{\alpha}}}$. For a Laurent series $p \in \R_{\ge 0}((x_1, \cdots, x_n))$, capacity is defined as:
    $$
           \cpc_{\bm{\alpha}}(p) = \begin{cases}
               \inf_{\bm{x} \in \Omega_p} \frac{p(\bm{x})}{\bm{x^{\alpha}}} & \text{ if } \Omega_p \neq \emptyset \\
               \infty & \text{ otherwise}
           \end{cases} 
           \hspace{0.2cm},
    $$
    and can also be written as $\cpc_{\bm{\alpha}}(p) = \inf_{\bm{x}> \bm{0}}\frac{p(\bm{x})}{\bm{x^{\alpha}}}$. Capacity of a Laurent series is always nonnegative, and is finite if and only if $\Omega_p \neq \emptyset$.
\end{definition}
            
Let us first present the main result of this section, \Cref{a lower bound for coefficients using capacity}, which states a lower bound for the coefficients of DL Laurent series using the capacity function. We then set up the proof of this theorem from \Cref{cap = 0 iff alpha not in newt p} to \Cref{base case for induction proving a lower bound for coefficients }. The proof of \Cref{a lower bound for coefficients using capacity} is presented at the end of this section. \Cref{a lower bound for coefficients using capacity} will be used in \S \ref{section: Flows of a graph} and \S \ref{section: applying bounds to parabolic verma modules} to obtain bounds for integer flows and dimension of weight spaces of parabolic Verma modules.
            \begin{theorem}
\label{a lower bound for coefficients using capacity}
    Let $p \in \R_{\ge 0}((x_1, \cdots, x_n))$ be a $d$-homogeneous DL Laurent series and $\bm{\alpha} \in \supp(p)$. Assume further that $\Omega_p$ is nonempty, and let $b_i \in \Z$ be an upper or lower bound for $\Deg_{x_i}^{\bm{\alpha}}(p)$ for all $i \ge 2$. Then:
    \[ 
    \frac{p_{\bm{\alpha}}}{\cpc_{\alpha}(p)} \ge \prod_{i = 2}^n \frac{|b_i - \alpha_i|^{|b_i - \alpha_i|}}{(1 + |b_i - \alpha_i|)^{1 + |b_i - \alpha_i|}}.
    \]
\end{theorem}

Suppose $p$ is a DL polynomial, then $0$ is a lower bound for $\Deg_{x_i}^{\bm{\alpha}}(p)$, and the degree of $x_i$ in $\restr{\partial_{i+1}^{\alpha_{i+1}}\cdots \partial_n^{\alpha_n}p}{x_{i+1} = \cdots = x_n = 0}$ is an upper bound for this set, allowing us to retrieve \Cref{theorem from BLP} from \Cref{a lower bound for coefficients using capacity}.

Note that \Cref{a lower bound for coefficients using capacity} requires all the sets $\Deg_{x_i}^{\bm{\alpha}}(p)$ to be either bounded above or below for $i \ge 2$. All DL polynomials $p \in \R[\bm{x}]$ and ``shifted" DL polynomials $\bm{x^{\alpha}}p(\bm{x})$ satisfy this condition. It might seem that our assumptions for this theorem are too restrictive, and that any Laurent series satisfying those assumptions is either a DL polynomial or a shifted DL polynomial. But we will see in the next section that there are other DL Laurent series that fulfill this condition. In fact, any DL Laurent series whose support is contained in a proper cone satisfies \Cref{a lower bound for coefficients using capacity}'s assumptions.

In the remainder of this section, we study properties of the capacity function that will be used to prove the lower bound in \Cref{a lower bound for coefficients using capacity}. In particular, we focus on bivariate DL Laurent series, study their capacity function, and then prove \Cref{a lower bound for coefficients using capacity} for such Laurent series. This result will be the base case for an inductive argument. The proof of \Cref{a lower bound for coefficients using capacity} is then completed using the recursive nature of domains of convergence from \Cref{omega pk and how it is related to omega p}. 

Before restricting to bivariate Laurent series, let us determine when $\cpc_{\bm{\alpha}}(p)$ is non-zero, that is, when the left hand-side of the inequality in \Cref{a lower bound for coefficients using capacity} is a real number. 

    \begin{lemma}
    \label{cap = 0 iff alpha not in newt p}
        Assume that $p \in \R_{\ge 0}((x_1, \cdots, x_n))$ is a nonzero Laurent series with a non-empty domain of convergence. Then $\cpc_{\bm{\alpha}}(p) > 0$ if and only if $\bm{\alpha}\in \Newt(p) $, where $\Newt(p)$ is the convex hull of $\supp(p)$ in $\R^n$.
    \end{lemma}
    \begin{proof}
           The statement holds for all polynomials $p \in \R[\bm{x}]$ by {\cite[Fact 2.18]{AO17}}. Their argument can be generalized to Laurent series as follows.
           \\
           Let $\bm{\alpha} \in \Newt(p)$. Then there exists some $k \in \N$ and $\bm{d}_1, \cdots, \bm{d}_k \in \supp(p)$ satisfying $\bm{\alpha} = \lambda_1 \bm{d}_1 + \cdots + \lambda_k \bm{d}_k$ for some $\lambda_1 + \cdots + \lambda_k = 1$ and $\lambda_i > 0$. For any $\bm{x}> \bm{0}$ we have: 
        \[
        p(\bm{x}) \ge \sum_{i = 1}^k p_{\bm{d_i}} \bm{x^{\bm{d_i}}} = \sum_{i = 1}^k \lambda_i (\frac{p_{\bm{d_i}}}{\lambda_i} \bm{x^{\bm{d_i}}}) \ge \prod_{i = 1}^k (\frac{p_{\bm{d_i}}}{\lambda_i} \bm{x^{\bm{d_i}}})^{\lambda_i} = \prod_{i = 1}^k \left ( \frac{p_{\bm{d}_i}}{\lambda_i}^{\lambda_i}\right)\bm{x}^{\bm{\alpha}},
        \]    
        and $\inf_{\bm{x}> \bm{0}} \frac{p(\bm{x})}{\bm{x^{\alpha}}} \ge \prod_{i=1}^n \left( \frac{p_{\bm{d}_i}}{\lambda_i}\right)^{\lambda_i}> 0$.

        Now assume that $\bm{\alpha} \not \in \Newt(p)$. Since $\Newt(p)$ is a closed convex set, there exists a separating hyperplane $\bm{c}\in \R^n$ satisfying $\inf_{\bm{d} \in \supp(p)} \bm{c}^T({\bm{\alpha} - \bm{d}}) = \epsilon> 0$. Let $\bm{x} \in \Omega_p$ and $\bm{y}_t = (x_1e^{tc_1}, \cdots , x_ne^{tc_n})$ for all $t \in \R$. Then $\bm{y}_t \in \Omega_p$ by \Cref{scaling the variables}, and:
        \begin{align*}
            \cpc_{\bm{\alpha}}(p) \le \inf_{t \in \R} \frac{p(\bm{y}_t)}{\bm{y^{\alpha}}_t}
            &= \inf_{t \in \R}\sum_{\bm{d} \in \supp(p)} p_{\bm{d}} e^{t \bm{c}^T(\bm{d}- \bm{\alpha})}\bm{x^{d-\alpha}} \\
            & \le \inf_{t \in \R}e^{\epsilon t} \sum_{\bm{d} \in \supp(p)} p_{\bm{d}}\bm{x^{d-\alpha}}
            \\&=\frac{p(\bm{x})}{\bm{x^{\alpha}}} \inf_{t \in \R}e^{\epsilon t} = 0,
        \end{align*}
        and therefore $\cpc_{\bm{\alpha}}(p) = 0$.     
    \end{proof}

        Since the support of a DL Laurent series is $M$-convex by \Cref{support of DL Laurent series is M convex}, its Newton polytope is exactly its support. Hence, by \Cref{cap = 0 iff alpha not in newt p}, the assumption that $\bm{\alpha} \in \supp(p)$ implies that both $p_{\bm{\alpha}}$ and $\cpc_{\bm{\alpha}}(p)$ are nonzero, and that the left hand-side of the inequality in \Cref{a lower bound for coefficients using capacity} is a real number. In fact, one could verify \Cref{cap = 0 iff alpha not in newt p} for the bivariate DL Laurent series $\sum_{m \ge n} y^m x^{d-m}$ by hand.

    \begin{example}
    \label{capacity function for 1/x-y and 1/y-x}
    Let $n,d \in \Z$ be arbitrary. The domain of convergence of the Laurent series $p(x,y) = \sum_{m \ge n}y^mx^{d-m}$ is the set $\Omega_p = \{(x,y) : x>y>0 \}$, and for any $(x_0,y_0) \in \Omega_p$, $p(x_0,y_0) = y_0^n{x_0^{-n-d}}(1 - \frac{y_0}{x_0})^{-1}$. Now we can compute the capacity of $p$ at point $(\bm{\alpha}, \bm{\beta})$ as follows:
        \begin{align*}
            \cpc_{(\alpha, \beta)}(p)
            & = \inf_{0 < y < x} \frac{1}{x^{\alpha + \beta + d}} \cdot \frac{(1-\frac{y}{x})^{-1}}{(\frac{y}{x})^{\beta - n}}\\
            & = \inf_{0 < x} \frac{1}{x^{\alpha + \beta + t}}\cdot \inf_{0 < z < 1}\frac{(1-z)^{-1}}{z^{\beta-n}}\\
            & = \begin{cases}
                0 & \text{ if } \alpha + \beta + t \neq 0\\
                \inf_{0 < z < 1} \frac{(1-z)^{-1}}{z^{\beta-n}} & \text{ if } \alpha + \beta +t = 0
            \end{cases}
        \end{align*}
         For any $k \in \R$, $\frac{d}{dz}\left ( \frac{(1-z)^{-1}}{z^k}\right )$ is only zero at point $z_0 = \frac{k}{1 + k}$. If $\frac{k}{k+1}\in (0,1)$, then it is easy to check that $\frac{k}{k+1} = \arg \min \left(\frac{(1-z)^{-1}}{z^{\beta}}\right)_{0 < z < 1}$. If $\frac{k}{k+1} \not \in (0,1)$ however, $\frac{(1-z)^{-1}}{z^{\beta}}$ is a monotone function in the interval $ (0,1)$, and: 
         \[\inf_{0 < z <1} \frac{(1-z)^{-1}}{z^{k}} = \min \{ \lim_{z \to 0^{+}} \frac{(1-z)^{-1}}{z^{k}}, \lim_{z \to 1^{-}} \frac{(1-z)^{-1}}{z^{k}} \}   = 0.
         \]
         Therefore:
         \begin{align*}
             \cpc_{(\alpha, \beta)}(p)
             & = \begin{cases}
                0 & \text{ if } \alpha + \beta + t \neq 0\\
               \frac{(1+ \beta - n)^{1 + \beta - n}}{{(\beta-n)}^{\beta-n}} & \text{ if } \alpha + \beta +t = 0 \text{ and } 0 \le \frac{\beta-n}{1 + \beta-n} < 1 \\
                0 &\text{otherwise}
                \end{cases}\\    
                & = \begin{cases}
                0 & \text{ if } (\alpha , \beta) \not \in \Newt(p)\\
               \frac{(1+ \beta - n)^{1 + \beta - n}}{{(\beta-n)}^{\beta-n}} & \text{ if } (\alpha , \beta) \in \Newt(p) \\
                \end{cases} \hspace{0.2cm}. 
         \end{align*}
         So for any $k \ge n$, $\cpc_{(t-k, k)}(\sum_{m \ge n}y^m x^{t-m}) = \frac{(1+k-n)^{1+k-n}}{(k-n)^{k-n}}$. By a similar argument, for any $k \le n$, one can prove that $\cpc_{(t-k,k)}(\sum_{m \le n} y^m x^{t-m}) = \frac{(1+n-k)^{(1+n-k)}}{(n-k)^{n-k}}$.
    \end{example}

        In what comes next, we will define weighted log-concavity for bivariate Laurent series, and then prove a lower bound for the coefficients of bivariate weighted log-concave Laurent series.

\begin{definition}[weighted log-concavity]
         Suppose we are given two Laurent series $p (x,y) = \sum_{m \in \Z} p_m x^{d-m} y^m$ and $w(x,y) = \sum_{m \in \Z} w_m x^{d-m} y^m$ such that $w$'s support does not have any holes (i.e, the sequence $\{w_m\}_{m \in \Z}$ does not have any internal zeroes). Say that $p$ is a \textit{$w$-log-concave} Laurent series if $\supp(p) \subseteq \supp(w)$ and the sequence $\{\frac{p_m}{w_m}\}_{m \in \Z, w_m \neq 0}$ is a log-concave sequence. 
\end{definition}
\begin{remark}
\label{the support of w-log-concave Laurent series does not have any holes}
    It follows by definition that the support of $p$ cannot have any internal zeros if $p$ is $w$-log-concave. 
\end{remark}

\begin{lemma}
\label{Omega w is non-empty -> Omega p is non-empty}
    Assume that $p$ is $w$-log-concave and $\Omega_w \neq \emptyset$. Then $\Omega_p$ is also non-empty.
\end{lemma}
\begin{proof}
    Let $p(x,y) = \sum_{m \in \Z} p_m x^{d-m}y^m$ and $w(x,y) = \sum_{m \in \Z} w_m x^{d-m}y^m$, and let $(x_0,y_0) \in \Omega_w$. we will find $s,r \in \R_{>0}$ such that $(sx_0, ry_0) \in \Omega_p$.
    
    Suppose $\{ \frac{p_m}{w_m} \}_{m \in \Z, p_m \neq 0}$ is a geometric sequence. So there exists some $c,a \in \R_{> 0}$ satisfying $p_m = cw_m a^m$ for all non zero $p_m$, and:
    \[
    p(x_0, \frac{1}{a}y_0) = \sum_{m \in \Z} \frac{p_m}{a^m}x_0^{d-m}y_0^m = c w(x_0, y_0).
    \]
    Therefore, $(x_0, \frac{1}{a}y_0) \in \Omega_p$. 
    
    Now assume that $\{ \frac{p_m}{w_m}\}_{m \in \Z, p_m \neq 0} $ is not a geometric sequence. So the Laurent series $q(x,y) := \sum_{m \in \Z, w_m \neq 0} \frac{p_m}{w_m}x^{d-m}y^m$ is non-trivially DL, and by \Cref{domain of convergence is non-empty for non trivial DL bivariate Laurent series}, $\Omega_q$ is non-empty. Take an arbitrary point $(x_1, y_1) \in \Omega_q$ and let $L := q(x_1, y_1)$. Then for any $m \in \Z$, each monomial $\frac{p_m}{w_m}x_1^{d-m}y_1^m $ is less than $ L$, and $p_m (x_0x_1)^{d-m} (y_0y_1)^m$ which is a monomial od $p(x_0x_1,y_0y_1$, is less than its corresponding monomial $Lw_m x_0^{d-m}y_0^m$ in $Lw(x_0,y_0)$. We can now use the comparison test to show that $p(x_0x_1, y_0y_1)$ is a converging sequence since $p(x_0x_1,y_0y_1) \le L w(x_0,y_0) $.
\end{proof}

\begin{lemma} [{\cite[Lemma 5.7]{BLP23}}]
\label{coefficient bound for w-log concave Laurent series}  
    Let $p$ be $w$-log-concave Laurent series given by $p(x,y) = \sum_{m \in \Z} p_m x^{d-m}y^m$ and $w(x,y) = \sum_{m \in \Z} w_m x^{d-m}y^m$. For any $k\in \Z$ we have: 
    \[
    \frac{p_k}{\cpc_{(d-k, k)}(p)} \ge \frac{w_k}{\cpc_{(d-k, k)}(w)}.
    \]\end{lemma}
\begin{proof}
    The right hand side of this inequality will be zero if $\Omega_w = \emptyset$, and by \Cref{cap = 0 iff alpha not in newt p}, the left hand side of the inequality is $+\infty$ if $(d-k,k) \not \in \Newt(p)$. Since the support of $p$ does not have internal zeros by \Cref{the support of w-log-concave Laurent series does not have any holes}, $(d-k,k)$ is a lattice point of $\Newt(p)$ if and only if $(d-k, k) \in \supp(p)$. So the statement of this lemma is trivial if $\Omega_w = \emptyset$, or if $(d-k,k) \not \in \supp(p)$, and we will assume that neither are the case for the rest of this proof. 

    We have that $(d-k,k) \in \supp(p) \subseteq \supp(w)$. Again, the statement is trivial if $\supp(w) = \{ (d-k,k) \}$, since we would have $\cpc_{(d-k,k)}(p) = p_k$ and $\cpc_{(d-k,k)}(w) = w_k$. So we can assume that $(d-k,k)$ is not the only element of $\supp(w)$, and either $w_{k+1}$ or $w_{k-1}$ is nonzero as well. Moreover, since $\Omega_w \neq \emptyset$, by \Cref{Omega w is non-empty -> Omega p is non-empty}, $\Omega_p$ is non-empty. 
   
    Define:
    \[
    C_{d-k,k} := \sup_{l.c. \bm{a}} \inf_{x,y>0}  \left( \frac{1}{a_k x^{d-k} y^{t}}\sum_{(d-j,j) \in \supp(w)} w_j a_j x^{d-j}y^j \right).
    \]
    Where the $\sup$ is over all positive log concave sequences $\bm{a} = \{a_m\}_{m \in \Z, w_m \neq 0}$. The statement of the lemma is equivalent to proving that $C_{d-k,k} = \cpc_{(d-k,k)}(w)$. The rest of this proof is similar to the proof of \cite[Lemma 5.7]{BLP23} with minor changes. We can rewrite $C_{d-k,k}$ as follows:
    \begin{align*}
        C_{d-k, k } & = \sup_{\substack{l.c. \bm{a} \\ a_k = 1}} \inf_{z > 0} \left( \sum_{(d-j,j) \in \supp(w)} w_j a_j z^{j-k} \right),
    \end{align*}
    and we can further restrict to the case where $a_{k-j} = a_{k-1}^j$ and $a_{k + j} = a_{k + 1}^j$ for any $j \ge 1$. Thus:
    \[
    C_{d-k,k} = \sup_{a_{k-1}.a_{k+1}\le 1} \inf_{z > 0} \left[\left( \sum_{\substack{j \le k - 1 \\ (d-j,j) \in \supp(w)}}w_{j} a_{k-1}^{k - j}z^{j-k}\right)+ w_k + \left( \sum_{\substack{j \le k + 1 \\ (t-j,j) \in \supp(w)}} w_{j} a_{k + 1}^{j - k} z^{j - k} \right) \right].
    \]
    As discussed before, either $w_{k+1}$ or $w_{k-1}$ is nonzero. It is easy to check that the supremum will not be attained for the sequence $\bm{a}$ if $a_{k-1} = a_{k+1}=0$ . Without loss of generality, let $a_{k+1} \neq 0$. So $a_{k-1}\le a_{k + 1}^{-1}$, which implies:
    \[
    C_{d-k , k} = \sup_{a_{k+1} > 0} \inf_{z>0} \left( \sum_{(d-j,j) \in \supp(w)} w_j (a_{k+1}z)^{j-k}\right) = \inf_{z > 0} \left( \sum_{(d-j,j) \in \supp(w)} w_j z^{j-k}\right) = \cpc_{(d-k,k)}(w),
    \]
    proving the result. 
\end{proof}

Let $p(x,y) = \sum_{m \in \Z} p_m x^{d-m}y^m$ be a bivariate DL Laurent series. Then $p$ is $\sum_{m \in S} x^{d-m}y^m$-log-concave for any $S \supseteq \supp(p)$. Applying \Cref{coefficient bound for w-log concave Laurent series} to this special case gives us the following lower bound for the coefficients of bivariate DL Laurent series.
\begin{corollary} 
\label{a bound for coefficients of bivariate DL Laurent series}
    Let $p(x,y) =\sum_{m \in \Z} p_m x^{d-m} y^m$ be a DL Laurent series. Assume that $n$ is either a lower bound or an upper bound for the set $\{m : (d-m,m) \in \supp(p) \}. $Then:
    \[ 
    p_k \ge{\cpc_{(d-k,k)}(p)}. \frac{|k-n|^{|k-n|}}{(1+|k-n|)^{1+|k-n|}},
    \]
    for any $k$ such that $(d-k,k) \in \supp(p)$.
\end{corollary}
\begin{proof}
    If $p(x,y) = \sum_{m \ge n} p_m x^{d-m}y^m$ is DL, then $p$ is ($\sum_{m \ge n} x^{d-m} y^m$)-log-concave. Using \Cref{capacity function for 1/x-y and 1/y-x} and \Cref{coefficient bound for w-log concave Laurent series}, we have:
    \[
    \frac{p_k}{\cpc_{(d-k, k)}(p)}\ge \frac{(k-n)^{k-n}}{(1+k-n)^{1+k-n}},
    \]
    for any $k \ge n$. Note that $p$ is also ($\sum_{m \le d-n} x^m y^{d-m}$)-log-concave, but \Cref{coefficient bound for w-log concave Laurent series} gives the same bound as above.
    
    The statement can be proven similarly for $p(x,y) = \sum_{m \le n} p_m x^{d-m}y^m$. 
\end{proof}

\begin{remark}
    We cannot use the same method to bound the coefficients of a general DL Laurent series $p(x,y) = \sum_{m \in \Z} p_mx^{d-m}y^{m}$. The best thing we can say about this Laurent series is that it is $w$-log-concave for $w(x,y) = \sum_{m \in \Z}x^{d-m}y^m$, and using \Cref{a bound for coefficients of bivariate DL Laurent series} for this $p$ and $w$ gives us a trivial bound. We can try and fix this by swapping $w$ for a Laurent series that captures the ``growth" of $p$'s coefficients, but that would beat the purpose of bounding the coefficients of a Laurent series without knowing much about the coefficients themselves. 
\end{remark}

Let us rewrite \Cref{a bound for coefficients of bivariate DL Laurent series} using the set $\Deg_{x_i}^{\bm{\alpha}}$ from \Cref{degree set}. 
\begin{corollary}
\label{base case for induction proving a lower bound for coefficients }
    Let $p \in \R((x_1,x_2))$ be a DL Laurent series and let $(\alpha_1, \alpha_2) \in \supp(p)$. Further assume that $b_2$ is either a lower or an upper bound for $\Deg_{x_2}^{(\alpha_1, \alpha_2)}(p)$, then we have: 
    \[
    p_{(\alpha_1, \alpha_2)} \ge \cpc_{(\alpha_1, \alpha_2)}(p)\cdot \frac{(|\alpha_2 - b_2|)^{|\alpha_2 - b_2|}}{(1+|\alpha_2 - b_2|)^{1+|\alpha_2 - b_2|}}.
    \]
\end{corollary}

Now we have all the tools needed to prove \Cref{a lower bound for coefficients using capacity}.
\begin{proof}[Proof of \Cref{a lower bound for coefficients using capacity}]
    The proof is by induction over $n \ge 2$. The case $n =2$ is \Cref{base case for induction proving a lower bound for coefficients }. Let $n > 2$.
    First, assume that $\Deg_{x_n}^{\bm{\alpha}}(p)$ is bounded below by $b_n$, and write $p(x_1, \cdots, x_n) = \sum_{k = b_n}^{\infty} x_n^{k} p_k(x_1, \cdots, x_{n-1}).$ 
    
    For any $(z_1, \cdots, z_{n-1}, z_n) \in \Omega_p$, we know that $(y_1t, \cdots, y_{n-1}t, t) \in \Omega_p$ by \Cref{scaling the variables}, where $y_i = \frac{z_i}{z_n}$ and $t > 0$ is arbitrary, besides, for any $s \in (0,t)$:
    \begin{align*}
     p(y_1t, \cdots, y_{n-1}t, s) &= \sum_{k \ge b_n} s^k t^{d-k} p_k(y_1, \cdots, y_{n-1})\\
     &= s^{b_n} \sum_{k \ge b_n} s^{k-b_n} t^{d-k} p_k(y_1, \cdots, y_{n-1})\\
     &\le s^{b_n} \sum_{k \ge b_n} t^{d-b_n}p_k(y_1, \cdots, y_{n-1}) \\
     &= (\frac{s}{t})^{b_n} p(y_1t, \cdots, y_{n-1}t , t) < \infty.
    \end{align*}
    Note that each $p_k$ is a ($d-k$)-homogeneous Laurent series with $(y_1, \cdots, y_{n-1})$ in its domain of convergence by \Cref{any point in the domain of convergence of a p gives a log concave sequence pk}. So $p(y_1t, \cdots, y_{n-1}t, s)$ also converges for any $0 < s < t$. Observe that the capacity of $p$ at point $\bm{\alpha}$ is at most $\frac{p(y_1t,\cdots, y_{n-1}t,s)}{(y_1t,\cdots, y_{n-1}t,s)^{\bm{\alpha}}}$, and we can write:
    \begin{align*}
        \cpc_{\alpha}(p) &\le \inf_{0<s<t}\frac{\sum_{k \ge b_n} s^k t^{k-b} p_k(y_1, \cdots, y_{n-1})}{y_1^{\alpha_1}\cdots y_{n-1}^{\alpha_{n-1}}s^{\alpha_n}t^{d-\alpha_n}}\\
        &=\frac{1}{y_1^{\alpha_1} \cdots y_{n-1}^{\alpha_{n-1}}}\inf_{0 < s < t} \frac{\sum_{k \ge b_n}s^k t^{d-k}p_k(y_1,\cdots, y_{n-1})}{s^{\alpha_n} t^{d- \alpha_n}}.
    \end{align*}
    The sequence $\{ p_k(y_1, \cdots, y_{n-1})\}_{k \ge b_n}$ is log-concave by \Cref{omega pk and how it is related to omega p}, so the Laurent series \\ $\sum_{k \ge b_n}s^k t^{d-k}p_k(y_1,\cdots, y_{n-1})$ is a bivariate DL Laurent series in $s$ and $t$. Using the induction hypothesis for $n=2$, we obtain:
    \begin{align*}
         \cpc_{\alpha}(p) &\le \frac{1}{y_1^{\alpha_1} \cdots y_{n-1}^{\alpha_{n-1}}}\inf_{0 < s < t} \frac{\sum_{k \ge b_n}s^k t^{d-k}p_k(y_1,\cdots, y_{n-1})}{s^{\alpha_n} t^{d- \alpha_n}}\\
        & \le  \frac{p_{\alpha_n}(y_1,\cdots,y_{n-1})}{y_1^{\alpha_1} \cdots y_{n-1}^{\alpha_{n-1}}}\cdot\frac{(1 + \alpha_n - b_n)^{1 + \alpha_n - b_n}}{(\alpha_n - b_n)^{\alpha_n - b_n}} \\
        &= \frac{p_{\alpha_n}(z_1, \cdots, z_{n-1})}{z_1^{\alpha_1}\cdots z_{n-1}^{\alpha_{n-1}}}\cdot \frac{(1 + \alpha_n - b_n)^{1 + \alpha_n - b_n}}{(\alpha_n - b_n)^{\alpha_n - b_n}}.
    \end{align*}
    Since this inequality holds for any $(z_1, \cdots, z_n) \in \Omega_p$, we conclude:
    \begin{align*}
         \cpc_{\alpha}(p) &\le \inf_{(z_1, \cdots, z_{n-1},z_n) \in \Omega_p} \frac{p_{\alpha_n}(z_1,\cdots,z_{n-1})}{z_1^{\alpha_1} \cdots z_{n-1}^{\alpha_{n-1}}}\cdot\frac{(1 + \alpha_n - b_n)^{1 + \alpha_n - b_n}}{(\alpha_n - b_n)^{\alpha_n - b_n}} \\ 
        &= \inf_{(z_1, \cdots, z_{n-1}) \in \Omega_{p_{\alpha_n}}}\frac{p_{\alpha_n}(z_1,\cdots,z_{n-1})}{z_1^{\alpha_1} \cdots z_{n-1}^{\alpha_{n-1}}} \cdot\frac{(1 + \alpha_n - b_n)^{1 + \alpha_n - b_n}}{(\alpha_n - b_n)^{\alpha_n - b_n}}\\
        & = \cpc_{(\alpha_1,\cdots,\alpha_{n-1})}(p_{\alpha_n})\cdot\frac{(1 + \alpha_n - b_n)^{1 + \alpha_n - b_n}}{(\alpha_n - b_n)^{\alpha_n - b_n}}.
    \end{align*}
   The second to the last step is due to \Cref{projection of p's domain of convergence}. The results follows from the induction hypothesis for $p_{\alpha_n}$. 

   If $\Deg_{x_n}^{\bm{\alpha}}$ is bounded above, then a similar argument works. 
\end{proof}
 
\section{Integer flows of a graph}
\label{section: Flows of a graph}
Recall that an integer flow of a directed graph $G = (V, E)$ is some $\phi : E \to \Z_{\ge 0}$. For any vertex $v$ and such a flow $\phi$, we say that the net-flow of $v$ with respect to $\phi$ is $\sum_{\overset{e}{\to} v} \phi(e) - \sum_{v \overset{e}{\to}} \phi(e)$ and is denoted by $\netflow_{\phi}(v)$. We will also regard the map $\netflow: V \to \Z$ as a vector $\netflow_{\phi} \in \Z^{V}$. Furthermore, for some $\bm{N} \in \Z^V$, let $K_G(\bm{N})$ be the set of flows of $G$ with net-flow $\bm{N}$. 

Consider the generating series $f_G(\bm{x}) := \sum_{\phi: E \to \Z_{\ge 0}} \bm{x}^{\netflow_{\phi}}$. In case $G$ is acyclic, this is a well-defined Laurent series in which the coefficient of $\bm{x^{N}}$ counts $|K_G(\bm{N})|$. In \Cref{subsection: definition of flow series}, we will argue that this series is indeed a DL Laurent series when $G$ is acyclic. We then apply \Cref{a lower bound for coefficients using capacity} to obtain capacity bounds for the number of integer flows with a fixed net-flow in \Cref{subsection: bounds for integer flows}. Using convex analysis, we rewrite $\cpc_{\bm{N}}(f_G)$ as a maximization problem. Evaluating this maximization problem at any points yields a more explicit lower bound, which we have stated in \Cref{subsection: explicit bounds for integer flows}.

Finally, we consider two special case in \Cref{subsection: special cases}:
\begin{itemize}
    \item when $G$ is a complete acyclic graph, we recover bounds for type A Kostant partition functions,
    \item when $G$ is a complete bipartite acyclic graph, we recover bounds for the number of contingency tables.
\end{itemize}

\subsection{The flow Laurent series of a graph}
\label{subsection: definition of flow series}
Let us take a quick look at the generating series $\sum_{\phi: E \to \Z_{\ge 0}} \bm{x}^{\netflow_{\phi}}$ and introduce a closed form for it.

\begin{definition}[flow Laurent series]
    \label{Laurent series assigned to a graph}
Let $G = ([n], E)$ be a directed graph. For any edge $v \overset{e}{\to} u \in E$, define:
\[
f_{e}(x_v,x_u) = \sum_{k \ge 0} x_u^k x_v^{-k}
\]
and:
\[f_G(\bm{x}) = \prod_{e = v \to u \in E} f_e(x_v,x_u)\]
\end{definition}

\begin{remark}
    \label{acyclic graphs have DL Laurent series}
    Assume that $G$ is an acyclic graph. For any edge $e$, the series $f_e$ is a DL Laurent series by \Cref{bivaraite Laurent series is DL iff it has a log concave sequence}. It is easy to check that when $G$ is acyclic, $f_G$ is a well-defined product of such DL Laurent series and by \Cref{multiplication of DL Laurent series with admissible pairs}, is a DL Laurent series itself. In fact, $f_G$ is the closed form of the generating series $\sum_{\phi: E \to \Z_{\ge 0}}\bm{x}^{\netflow_{\phi}}$.
\end{remark}
Throughout the rest of this paper, unless stated otherwise, we will assume that all of our directed graphs are simple, have vertices labeled with integers, and that all edges are directed from vertices with larger labels to smaller ones. 
\begin{remark}
\label{domain of convergence of flow Laurent series}
        Let $G$ be an acyclic directed graph. Then: $$\Omega_{f_G} = \{ (x_1, \cdots, x_n) \in 
        \R_{\ge 0}: x_v > x_u > 0 \text{ for all } e = u \to v \in E(G)\}, $$
        and for any $\bm{x} \in \Omega_{f_G}$: 
        $$f_G(\bm{x}) = \prod_{v \to u \in E(G)} \frac{1}{1 - \frac{x_v}{x_u}}.$$
\end{remark}
\subsection{Bounds for integer flows of a graph}
\label{subsection: bounds for integer flows}
Having established that the flow Laurent series of a directed acyclic graph $G$ is DL, we can now apply \Cref{a lower bound for coefficients using capacity} to this series. This yields a lower bound for the number of integer flows with a fixed net-flow, stated in \Cref{a lower bound for flows of any graph}. This lower bound uses the set $\Deg_{x_i}^{\bm{\alpha}}(f_G)$ and the capacity function. We will therefore turn towards expressing both the set $\Deg_{x_i}(\bm{\alpha})$ and the capacity function based on combinatorial properties of the graph.

\begin{theorem}
\label{a lower bound for flows of any graph}
    Let $G = ([n], E)$ be any acyclic directed graph. For any $\bm{N} \in \Z^n$ satisfying $|\bm{N}|_1 = 0$:
    \[
    |K_G({\bm{N}})| \ge \cpc_{\bm{N}}(f_G)\cdot \prod_{i = 2}^{n} \frac{|m_i -N_i|^{|m_i- N_i|}}{(|m_i- N_i| + 1)^{|m_i- N_i| + 1}},
    \]
    where $m_i \in \Z$ is either a lower bound or an upper bound for the following set:
    \[
    \Deg_{x_i}^{\bm{N}}(f_G) =  \{ \netflow_{\phi}(i) \mid  \phi: E \to \Z_{\ge 0}, \netflow_{\phi}(j) = N_j \hspace{.2cm} \forall j > i\}.
    \]
\end{theorem}
\begin{proof}
    Note that such $m_i$ is either a lower bound or an upper bound for the set $\Deg_{x_i}^{\bm{N}}(f_G)$.
     The statement follows from \Cref{a lower bound for coefficients using capacity}, \Cref{domain of convergence of flow Laurent series} and  \Cref{acyclic graphs have DL Laurent series}.
\end{proof}

Let us find explicit bounds for the set $\Deg_{x_i}^{\bm{N}}(f_G)$. \Cref{proposition to simplify the lower bound for Mi} bounds the elements of $\Deg_{x_i}^{\bm{N}}(f_G)$ below, and \Cref{a tight bound for Mi} bounds them above. Recall that we are assuming edges are directed from larger vertices to smaller ones. 

\begin{proposition}
    \label{proposition to simplify the lower bound for Mi}
    A terminal vertex in a directed graph $G = ([n], E)$ is a vertex $i \in [n]$ such that no edge $i \to j$ exists in $G$. If $i$ is a terminal vertex, 0 is a lower bound for the set $\Deg_{x_i}^{\bm{N}}(f_G)$. If $i$ is not a terminal vertex, then $\Deg_{x_i}^{\bm{N}}(f_G)$ has no lower bound. 
\end{proposition}
\begin{proof}
    Let $i$ be a terminal vertex, then for any $\phi : E \to \Z_{\ge 0}$, we have: 
    \[
    \netflow_{\phi}(i) = \sum_{\overset{e}{\to}i} \phi (e) \ge 0.
    \]
    If $i$ is not a terminal vertex, let $i \overset{e}{\to} v \in E$ be an edge. Assign integers to all edges but $e$ so that the net-flow of all vertices $j> i$ is $N_j$. Then we can assign an arbitrarily large flow to $e$, making the net-flow of $i$ arbitrarily small. 
\end{proof}

\begin{proposition}
\label{a tight bound for Mi}
    For a directed acyclic graph $G= ([n],E)$, let $H$ be the undirected induced subgraph of $G$ on $\{i, \cdots, n\}$. Within $H$, let $C_i$ be the connected component of $i$. Then $-\sum_{j \in C_i \setminus \{i\}}N_j$ is a tight upper bound for the set $\Deg_{x_i}^{\bm{N}}(f_G)$.
\end{proposition}
\begin{proof}
    For the sake of brevity, let $D_i = C_i \setminus\{i\}$. First, let us prove that $-\sum_{j \in D_i}N_j$ is an upper bound. Observe that there are no edges $e \in E$ entering the set $D_i$. Moreover, if $e : u \to v$ exits the set $D_i$, we must have $v \le i$. Hence:
    \[
        - \sum_{j \in D_i} N_j =\sum_{D_i \overset{e}{\to}}\phi(e)- \sum_{\overset{e}{\to}D_i}\phi(e)  = \sum_{D_i \overset{e}{\to} [i]} \phi(e) \ge  \sum_{D_i \overset{e}{\to} i} \phi(e). \label{N(i) to i} \tag{1}
    \]
    Moreover, any edge $v \to i$ satisfies $v \in D_i$, therefore: 
    \[
        \netflow_{\phi}(i) = \sum_{\overset{e}{\to}i} \phi(e) - \sum_{i\overset{e}{\to}} \phi(e) 
        \le \sum_{\overset{e}{\to}i}\phi(e) =  \sum_{D_i \overset{e}{\to}i} \phi(e) 
        \label{i from N(i)} \tag{2}
    \]
    It follows that $- \sum_{j \in D_i}N_j$ is an upper bound from \ref{N(i) to i} and \ref{i from N(i)}.

    To show that $-\sum_{j \in D_i}N_j$ is a tight upper bound, take some flow $\phi: E \to \Z_{\ge 0}$ with $\netflow_{\phi}(j) = N_j$ for $j > i$. Note that such a $\phi$ exists. Let $u \in [i-1]$ be a vertex that is connected to the set $D_i \cup \{ i\}$ with an edge, say $e : v \to u$ for some $v \in D_i \cup \{ i\}$. Let $p$ be a (possibly empty) undirected path from $v$ to $i$. Walk through the pass $e \cup p$ and subtract/add $\phi(e)$ units of flow from/to each edge of $e \cup p$ so that the net-flow of none of the intermediate vertices change. Repeat this for all such vertices $v$ and all edges connecting them to $D_i \cup \{i\}$ to obtain a new flow $\phi'$. By the way of construction, the net-flow of none of the vertices $j > i$ changes, and $\netflow_{\phi}(j) = \netflow_{\phi'}(j) = N_j$. Furthermore, $\phi'(e) = 0$ for any edge $v \overset{e}{\to} D_i \cup \{ i\}$. Now both of the inequalities in \ref{N(i) to i} and \ref{i from N(i)} become equalities:
    \[
        -\sum_{j \in D_i}N_j= \sum_{D_i \overset{e}{\to}} \phi'(e) - \sum_{\overset{e}{\to} D_i} \phi'(e)= \sum_{D_i \overset{e}{\to}[i]} \phi'(e) = \sum_{D_i \overset{e}{\to}i} \phi'(e),
        \label{tight 1} \tag{3}
        \]
    and:
    \[
    \netflow_{\phi'}(i) = \sum_{\overset{e}{\to}i}\phi'(e) - \sum_{i \overset{e}{\to}}\phi'(e) = \sum_{\overset{e}{\to}i}\phi'(e) = \sum_{D_i \overset{e}{\to}i}\phi'(e). \label{tight 2} \tag{4}
    \]
    Tightness of the bound follows from \ref{tight 1} and \ref{tight 2}.
\end{proof}

Now that we have a lower bound and an upper bound for the elements of the set $\Deg_{x_i}^{\bm{N}}(f_G)$, we can rewrite \Cref{a lower bound for flows of any graph} without using the set $\Deg_{x_i}^{\bm{N}}(f_G)$ in the statement.

\begin{corollary}
\label{simplified lower bound for flows in a general directed graph}
    Let $G = ([n], E)$ be a directed graph with terminal vertices $T$ and assume $\bm{N} \in \Z_{\ge 0}^{n}$ satisfies $|\bm{N}|_1 = 0$. For each $i$, let $C_i$ be the connected component of the undirected induced subgraph of $G$ on $\{i, \cdots, n\}$ containing $i$. Then: 
    \[
    |K_G({\bm{N}})| \ge \cpc_{\bm{N}}(f_G)\cdot \prod_{i = 2}^{n} \max \left\{ \frac{|s_i|^{|s_i|}}{(|s_i| + 1)^{|s_i| + 1}},\frac{|N_i|^{|N_i|}}{(|N_i| + 1)^{|N_i| + 1}}\cdot\delta_{i \in T} \right\}
    \]
    where $s_i = - \sum_{j \in C_i} N_j$, and $\delta_{i\in T}$ is the indicator variable of $i$ being a terminal vertex.
\end{corollary}
\begin{proof}
    Follows from \Cref{a lower bound for flows of any graph}, \Cref{proposition to simplify the lower bound for Mi} and \Cref{a tight bound for Mi}.
\end{proof}

\subsection{Explicit bounds for integer flows}
\label{subsection: explicit bounds for integer flows}
The next step is to write $\cpc_{\bm{N}}(f_G)$ as a more explicit quantity of $G$.
We use tools from convex analysis to turn the minimization problem $\cpc_{\bm{N}}f_G = \inf_{\bm{x} \in \Omega_{f_G}} \frac{f_G}{\bm{x^N}}$ to a maximization problem. We can then evaluate this maximization problem at any point to obtain a lower bound for the capacity function $\cpc_{\bm{N}}f_G$. The ideas we use in the rest of this section have been presented before in \cite[Lemma 5]{Bar12}, \cite[Proposition 6.2]{BLP23}, and \cite[Proposition 3.1]{LM26}.

We will start with a short review of some basic concepts in convex analysis. A function $f: \R^n \to \R^*$ is said to be convex if its domain $D(f) := \{ \bm{x} \in \R^n: f(\bm{x}) < \infty \}$ is convex and $f$ is convex over $D(f)$. Further, the convex conjugate of $f$ is the function: 
    \[
    f^*(\bm{y}) = \sup_{\bm{x} \in D(f)} \left\{ \left<\bm{x}, \bm{\alpha}\right> - f(\bm{x})\right\},
    \]
where $\left<.,.\right>$ is the normal dot product over $\R^n$. 
\begin{theorem}[{\cite[Theorem 16.4]{Roc97}}] 
\label{conjugate of sum of convex functions}Let $f$ be a convex function given by the sum $f = \sum_i f_i$ of convex functions with respective domains $D_i$. If $\cap_i\relint D_i \neq \emptyset $, then:
\[
f^*(\bm{\alpha}) = \inf_{\overset{\sum_i \bm{\alpha}_i = \bm{\alpha}}{\bm{\alpha}_i \in D(f_i^*)}} f_i(\bm{\alpha}_i)
\]
    where for each $\bm{\alpha}$, the infimum is attained. 
\end{theorem}
Using \Cref{conjugate of sum of convex functions}, we can rewrite $\cpc_{\bm{N}}(f_G)$ as a maximization problem: 
\begin{proposition}
\label{convex analysis of capacity of flow series}
    Let $G = ([n],E)$ be an acyclic directed graph with edges directed from larger vertices to smaller vertices. Recall that $K_{G}(\bm{N}) \subseteq \Z^{E}_{\ge 0}$ is the set of integral $\bm{N}$-flows of $G$. Then:
    \[
    \cpc_{\bm{N}} (f_G) = \sup_{\bm{a} \in K_G(\bm{N})}\prod_{j \to i \in E(G)} \frac{(a_{ij}+ 1)^{a_{ij} + 1}}{a_{ij}^{a_{ij}}}.
    \]
   
\end{proposition}
\begin{proof}
    Let $f = - \log{f_G(e^{\bm{x}})} = - \sum_{j \to i \in E} \log (1 - e^{x_i -x_j})$. The function $-\log(1 - e^t)$ is convex over its domain $\{ t :  t<0\}$, so  $f$ is a convex function over its domain $D = \{\bm{x} \in \R^n: x_i < x_j \hspace{0.2cm} \forall j \to i \in E \}$, and is given as the sum of convex functions $f_{i,j}(\bm{x}) = - \log(1 - e^{x_i - x_j})$.  Let us calculate the conjugate of each $f_{i,j}$ where $j \to i$ is an edge in $G$ for some $j > i$.
    \begin{align*}
        f_{i,j}^*(\bm{\alpha}) &= \sup_{x_i < x_j} \left\{ \left< \bm{x}, \bm{\alpha}\right> - \log(1 - e^{x_i-x_j}) \right\}
    \end{align*}
    It is easy to see that $f_{i,j}^*(\bm{\alpha}) = \infty$ if $\alpha_k \neq 0$ for some $k \neq i,j$, and also, $f_{i,j}^*(\bm{\alpha}) = \infty$ if $\alpha_i \neq \alpha_j$. Hence any $\bm{\alpha} \in D(f_{i,j}^*)$ is of the form $\bm{\alpha} = c\cdot e_i - c. e_j$ for some $c \in \R$. Furthermore, we can use basic calculus to obtain: 
     \begin{align*}
        f_{i,j}^*(c\cdot e_i - c \cdot e_j) &= \sup_{x_i < x_j} \left\{ cx_i - cx_j - \log(1 - e^{x_i-x_j}) \right\} = \log \left( \frac{c^c}{(c + 1)^{c + 1}}\right). 
    \end{align*}
    By \Cref{conjugate of sum of convex functions}: 
    \[
    \sup_{\bm{x} \in D(f)} \left\{ \left<\bm{x}, \bm{\alpha}\right> - f(\bm{x})\right\} = \inf_{\sum_{j \to i \in E} \bm{\alpha}^{(i,j)} = \bm{\alpha}} \sum_{j \to i \in E} f^*_{i,j}(\bm{\alpha}^{(i,j)})
    \]
    Negating and exponentiating the above expression yields: 
    \[
    \inf_{\bm{x} \in D} \frac{f_G(e^{\bm{x}})}{(e^{\bm{x}})^{\bm{\alpha}}} = \sup_{\bm{a} \in K_G(\bm{\alpha})} \prod_{j \to i \in E} \frac{(a_{ij} + 1)^{a_{ij} + 1}}{a_{ij}^{a_{ij}}}.
    \] 
\end{proof}

Finally, we can replace the capacity function in \Cref{a lower bound for flows of any graph} to obtain the following lower bound for the number of integral flows of a directed acyclic graph.
\begin{corollary}
    \label{capacity swapped with a supermum for all graphs and complete graphs}
    Let $G= ([n], E)$ be a directed acyclic graph. Let $\bm{N} \in \Z^n$, then for any (not necessarily integer) flow $\phi$ of $G$:
    \[
    |K_G({\bm{N}})| \ge \left[  \prod_{e \in E(G)} \frac{(\phi(e) + 1)^{\phi(e) + 1}}{\phi(e)^{\phi(e)}}\right] \cdot\prod_{i = 2}^{n} \max \left\{ \frac{|s_i|^{|s_i|}}{(|s_i| + 1)^{|s_i| + 1}},\frac{|N_i|^{|N_i|}}{(|N_i| + 1)^{|N_i| + 1}}\cdot\delta_{i \in T} \right\},
    \]
    where $s_i = - \sum_{j \in C_i} N_j$, and $\delta_{i\in T}$ is the indicator variable of $i$ being a terminal vertex.
    
\end{corollary}
For a general directed acyclic graph $G$ and a general vector $\bm{N} \in \Z^V$, it is often not easy to find a canonical flow $\phi \in K_G(\bm{N})$. Thus, we will be mostly using \Cref{simplified lower bound for flows in a general directed graph} for applications.  

\subsection{Special cases}
\label{subsection: special cases}
Recall from \Cref{subsection: contingency tables} that if $G$ is a complete directed acyclic graph on vertices $[n+1]$, the number of integer flows of $G$ with a fixed net-flow $\bm{N}$ is given by the type A Kostant partition function, denoted by $K_n(\bm{N})$. We also recall that the the number of contingency tables with marginals $(\bm{\alpha}, \bm{\beta})$ is denoted by $\CT(\bm{\alpha}, \bm{\beta})$.

We first apply the lower bound from \Cref{simplified lower bound for flows in a general directed graph} to a complete graph to derive a bound for type A Kostant partition functions. We will then show that one could obtain $\CT(\bm{\alpha}, \bm{\beta})$ from a coefficient of the flow generating series of a bipartite graph. Applying \Cref{simplified lower bound for flows in a general directed graph} to this flow series then yields a bound for the number of contingency tables with given marginals. 

\begin{corollary}
    
\label{bounding flows in complete graphs}

Let $G = ([n+1], E)$ be a directed complete graph, and let $\bm{N} \in \Z_{\ge 0}^{n + 1}$ satisfy $|\bm{N}|_1 = 0$. Recall that $K_{n}(\bm{N})$ is the number of $\bm{N}$-flows of $G$. We have:
\[
      K_n(\bm{N})  \ge \cpc_{\bm{N}}(f_{G}). \prod_{i =1}^{n} \frac{|s_i|^{|s_i|}}{(|s_i| + 1)^{|s_i| + 1}},
     \]
where $f_G$ is the flow Laurent series of $G$, and $s_i = \sum_{k = 1}^i N_k$. 
\end{corollary}
\begin{proof}
 For each $i$, the connected component of the undirected induced subgraph of $G$ on $\{i , \cdots, n\}$ containing $i$ is $C_i = \{ i, i + 1, \cdots, n\}$. The result follows from \Cref{simplified lower bound for flows in a general directed graph} and the fact that $N_1 + \cdots + N_{n + 1} = 0$. 
\end{proof}

\begin{remark}
    \label{C is the generating series for contingency tables} Define:
\[
C (x_1, \cdots, x_n; y_1, \cdots, y_m) = \prod_{i = 1}^n \prod_{j = 1}^m \left( \sum_{k \ge 0} x_i^k y_j^{-k}\right)
\]
$C$ is the flow Laurent series of a complete bipartite graph on vertices $[n+m]$, where there is an edge from vertex $n+i$ to vertex $j$ for each $i \in [m], j \in [n]$. We have replaced variables $x_{n + 1}, \cdots, x_{n+m}$ with $y_1, \cdots, y_m$ for the sake of simplicity.

For a contingency tables $A\in \Z_{\ge 0}^{n \times m}$ of marginals $(\bm{\alpha}, \bm{\beta})$, choose the monomial $x_i^{A_{ij}}y_j^{- A_{ij}}$ from the sum $\sum_{k \ge 0} x_i^k y_j^{-k}$, to ultimately obtain $\prod_{i = 1}^n x_i^{\sum_{k = 1}^m A_{ik}}\prod_{j =1}^m y_j^{\sum_{k =1}^n A_{kj}} = \bm{x^{\alpha}} \bm{y}^{- \bm{\beta}}$. This gives us a bijection between the ways of getting $\bm{x^{\alpha}} \bm{y}^{-\bm{\beta}}$ in $C$ and the set of contingency tables of marginals $(\bm{\alpha}, \bm{\beta})$.
\end{remark}

\begin{corollary}
\label{main theorem for contingency tables}
    Let $\bm{\alpha} \in \Z_{\ge 0}^n, \bm{\beta} \in \Z_{\ge 0}^m$ and suppose $\bm{\alpha}$ is in decreasing order, that is, $\alpha_1 \ge \alpha_2 \ge \cdots$. Then:
    \[
    \CT(\bm{\alpha}, \bm{\beta}) \ge \cpc_{(\bm{\alpha}, -\bm{\beta})}(C).\prod_{j = 1}^m \frac{\beta_j^{\beta_j}}{(\beta_j+1)^{\beta_j + 1}} \prod_{i  = 2}^n \frac{\alpha_i^{\alpha_i}}{(\alpha_i+1)^{\alpha_i + 1}}.
    \]
\end{corollary}
\begin{proof}
We will apply \Cref{simplified lower bound for flows in a general directed graph} to $C(\bm{x};\bm{y})$ to find a bound for $[\bm{x}^{\bm{\alpha}}\bm{y^{- \beta}}]C(\bm{x}; \bm{y})$.

For each $j \in [m]$, vertex $n+j$ has no edges coming into it, so $C_{n+ j} = \{n + j \}$ and by \Cref{a tight bound for Mi}, $0$ is an upper bound for the set $M_{n + j}$. Note that $n+j$ is not a terminating vertex, so $M_{n + j}$ has no lower bound by \Cref{proposition to simplify the lower bound for Mi}.

For each $i \in [n]$, all the vertices $\{i + 1, \cdots, n + m \}$ are reachable from $i$ with an undirected path.
Therefore, $ - \sum_{k = n + 1}^m \beta_k + \sum_{k = i + 1}^n \alpha_k  = \sum_{k = 1}^i \alpha_k = s_i$ is an upper bound for the set $M_i$. Note that $i$ is a terminating vertex, and therefore, $l_i = 0$ is a lower bound for $M_i$. \Cref{simplified lower bound for flows in a general directed graph} gives us: 
\[
[\bm{x}^{\bm{\alpha}}\bm{y^{- \beta}}]C(\bm{x}; \bm{y}) \ge \cpc_{(\bm{\alpha},- \bm{\beta})} \prod_{i = 1}^m \frac{|\beta_i|^{|\beta_i|}}{(|\beta_i|+1)^|\beta_i|} \prod_{i  = 2}^n \max\left\{ \frac{|s_i - \alpha_i|^{|s_i - \alpha_i|}}{(|s_i - \alpha_i| + 1)^{|s_i - \alpha_i| + 1}}, \frac{|\alpha_i|^{|\alpha_i|}}{(|\alpha_i|+1)^{|\alpha_i| + 1}}\right\},
\]
where $s_i = \sum_{k = 1}^i \alpha_i$. Note that by assumption, $s_i -\alpha_i \ge \alpha_i$ for $i \ge 2$, and $\frac{z^z}{(1 + z)^{1 + z}}$ is decreasing on $\R_{>0}$, implying that: 
\[
\max\left\{ \frac{(s_i - \alpha_i)^{s_i - \alpha_i}}{(s_i - \alpha_i + 1)^{s_i - \alpha_i + 1}}, \frac{\alpha_i^{\alpha_i}}{(\alpha_i+1)^{\alpha_i + 1}}\right\} = \frac{\alpha_i^{\alpha_i}}{(\alpha_i+1)^{\alpha_i+ 1}},
\]
completing the proof. 
\end{proof}

\section{Parabolic Verma modules}
\label{section: applying bounds to parabolic verma modules}
Recall from \Cref{subsection: verma modules and parabolic verma modules} that given $n\in \N$, $J \subseteq [n]$ and some $\bm{\lambda} \in \Lambda_J^+$, the Verma module $M(\bm{\lambda})$ is some quotient of the universal enveloping of $\mathfrak{sl}_{n+1}(\C)$, and the parabolic Verma module $M(\bm{\lambda}, J)$ is some quotient of $M(\bm{\lambda})$. 

It is proven in \cite[Theorem 1.5]{KMS25} that the polynomial part of any shift of $\character M(\bm{\lambda}, J)$ is a DL polynomial, or equivalently, that this character is a DL Laurent series. In fact, by \cite[Section 3]{KMS25}, $\character M(\bm{\lambda}, J)$ is the product of a flow Laurent series and some Schur polynomials. In  \Cref{subsection: characters of parabolic Verma modules}, we analyze the structure of the graph whose flow Laurent series appears in this character, as well as the partitions indexing the corresponding Schur polynomials. Since $\character M(\bm{\lambda}, J)$ is DL, we obtain capacity bounds for the coefficients of this Laurent series using \Cref{a lower bound for coefficients using capacity}. These bounds are stated in \Cref{subsection: bounds for dimensions of weight spaces}. Lastly, we rewrite $\cpc_{\mu}(\character (M(\bm{\lambda},J))$ as a maximization problem in \Cref{subsections: explicit bounds for dimensions of weight spaces}. We then evaluate this maximization problem at an arbitrary point to get a more explicit lower bound for the dimension of weight spaces of parabolic Verma modules. 

\subsection{Characters of parabolic Verma modules}
\label{subsection: characters of parabolic Verma modules}
Fix $n \in\N$, $J \subseteq[n]$ and $\bm{\lambda} \in \Lambda_J^+$. By \cite[\S 3]{KMS25}, we can write: 
\[
\character M(\bm{\lambda},J) = \prod_{\overset{i < j \in [n+1]}{i \to j \in E(G_J)}} \left( 1 + x_jx_i^{-1} + x_j^2 x_i^{-2} + \cdots\right) \cdot \prod_{t = 0}^l s_{\lambda_t}(x_i: i \in J_r),
\]
where $G_J$ is a graph on vertices $[n+1]$ and contains the edge $i \to j$ if $i<j$ and $[i,j-1] \not \subseteq J$. Moreover, $J_0, \cdots, J_l$ are the connected components of the Dynkin sub-diagram of $J \subseteq [n]$. In other words, $J_0, \cdots, J_l$ partitions $J$ into maximal contiguous intervals, and $\lambda_r = (\lambda_j - \lambda_{1 + \max J_r} : j \in J_r)$. We can rewrite this character as follows:
\[
\character M(\bm{\lambda},J)(x_i: i \in [n+1]) = f_{G_J}(x_i : i \in [n+1]). \prod_{t = 1}^l s_{\lambda_r}(x_i : i \in J_r).
\]
 It is now easy to see that $\character M(\bm{\lambda}, J)$ is a DL Laurent series since it is a well defined product of a DL Laurent series, $f_{G_J}$, with a DL polynomial, $\prod_{r =1}^l s_{{\lambda}_r}$. We can therefore apply \Cref{a lower bound for coefficients using capacity} to $\character M(\bm{\lambda}, J)$, to obtain a lower bound for the dimensions of weight spaces of $M(\bm{\lambda}, J)$. This lower bound is presented in \Cref{initial lower bound for parabolic verma modules}. The structure of the graph $G_J$ plays an essential role in this lower bound, so let us analyze the structure of this graph. 

\begin{lemma}
\label{G_J is multipartite}
   Define $0=:i_0 < i_1 < \cdots < i_r < i_{r+1} := n+1$ such that $J^c= \{ i_1 < i_2 < \cdots < i_r\}$, then the underlying undirected graph of $G_J$ is a complete multipartite graph with parts $[1,i_1], [i_1+1 , i_2], \cdots,[i_{r-1}+1, i_r ], [i_r+1,n+1]$. 
   Moreover, each of these parts corresponds to a maximal continuous interval of $J$, namely:
   \[
   \sqcup_{t = 0}^r [i_t + 1, i_{t+1}-1] = J
   \]
   For the rest of this section, we will assume $l = r$ and $J_t = [i_t + 1, i_{t + 1}- 1]$ for any $t \in [0,r]$. 
\end{lemma}
\begin{proof}
    Let $u<v$ be two vertices in different parts. Let $k$ be such that $u \le i_k < v$ (if multiple such $k$s exist, choose one). Then $i_k \in [u,v-1]$ and $[u,v-1] \not \subseteq J$. By definition, $u \to v \in E(G)$.

    The other statement of this remark is trivial.
\end{proof}
\subsection{Bounds for dimensions of weight spaces of parabolic Verma modules}
\label{subsection: bounds for dimensions of weight spaces}
Applying \Cref{a lower bound for coefficients using capacity} to $\character M(\bm{\lambda},J) $ gives us the following lower bound for the dimensions of weight spaces of parabolic Verma modules: 
\begin{theorem}
\label{initial lower bound for parabolic verma modules}
For any parabolic Verma module $M(\lambda, J)$ and a weight $\bm{\mu}$:    
    \begin{align*}
        \dim M(\bm{\lambda}, J)_{\bm{\mu}} &\ge \cpc_{\bm{\mu}}\left(\character M (\bm{\lambda}, J) \right) \prod_{i = 1}^{n}\frac{|m_i|^{|m_i|}}{(|m_i| + 1)^{|m_i| + 1}}
    \end{align*}
where where $m_i = -\sum_{k \in C_i} \mu_k$, and $C_i$ is the connected component containing $i$ in the undirected induced subgraph of $G_J$ on $[i,n+1]$.

\end{theorem}

\begin{proof}
For each $i \in [n]$, let $J_{r_i}$ be the connected component of the Dynkin sub-diagram of $J \subseteq[n]$ containing $i$. Denote by $\text{SSYT}_{J_{r_i}}(\lambda_{r_i})$ the set of semi-standard Young tableaux of shape $\lambda_{r_i}$ filled with labels from the set $J_{r_i}$. For any semi-standard Young tableau $T$, let $w_i(T)$ be the number of boxes labeled $i$ in $T$.
 
    $\character M(\bm{\lambda}, J)$ is a DL Laurent series since it is the product of a DL Laurent series ($f_{G_J}$) with a DL polynomial ($\prod_{t = 0}^r s_{\lambda_t}$). By \Cref{a lower bound for coefficients using capacity}, for any weight multiplicity $\bm{\mu}$, we have: 
    \[
    \dim M(\bm{\lambda},J)_{\bm{\mu}} \ge \cpc_{\bm{\mu}}\left( \character M(\bm{\lambda}, J) \right). \prod_{i = 1}^{n}\frac{|b_i - \mu_i|^{|b_i - \mu_i|}}{(|b_i - \mu_i| + 1)^{|b_i - \mu_i| + 1}},
    \]
    where $b_i$ is a bound for the following set: 
    \[
    \{\netflow_{\phi}(i) + w_i(T_{r_i}) \mid  \phi: E(G_J) \to Z_{\ge 0}, T_{r_i} \in \text{SSYT}_{J_{r_i}}(\lambda_{r_i}), \netflow_{\phi}(k) + w_k(T_{r_k})= \mu_k \forall k > i\}
    \] 

    With an argument similar to \Cref{proposition to simplify the lower bound for Mi}, we obtain: 
    \[
    \sum_{i \overset{e }{\to }k } \phi(e) \le \sum_{\overset{k \in C_i}{k > i}} \netflow_{\phi}(k),
    \]
    therefore,
    \[
    \netflow_{\phi}(i) \ge - \sum_{i\overset{e}{\to}k} \phi(e) \ge - \sum_{\overset{k \in C_i}{k > i}}\netflow_{\phi}(k).
    \]
    For all $k > i$, we are assuming that $\netflow_{\phi}(k) + w_k( T_{r_k} )= \mu_k$, and as a result, $\netflow_{\phi}(k) \le \mu_k$. Combining this with the inequality above, we get $\netflow_{\phi}(i) \ge - \sum_{k \in C_i - \{ i\}}\mu_k$. Hence $-\sum_{k \in C_i - \{ i\}} \mu_k$ is a lower bound for $\netflow_{\phi}(i) + w_i(T_{r_i})$ for any choice of $\phi$ and $T_{r_i}$, completing the proof.
\end{proof}
The lower bound presented in \Cref{initial lower bound for parabolic verma modules} depends on the connected components of the undirected induced subgraphs of $G_J$. Since $G_J$ is a complete multipartite graph by \Cref{G_J is multipartite}, the connected component containing $i$ in the induced subgraph of $G_J$ on vertices $[i,n+1]$ is empty for any $i \ge i_r +1 $, and is the whole set $[i,n+1]$ for any $i \le i_r$. Let us now rewrite \Cref{initial lower bound for parabolic verma modules}.  
\begin{corollary}    
\label{structure of G_J applied to bounds for verma modules}
For any parabolic Verma module $M(\bm{\lambda}, J)$ and a weight $\bm{\mu}$:    
    \begin{align*}
        \dim M(\bm{\lambda}, J)_{\bm{\mu}} &\ge \cpc_{\bm{\mu}}\left(\character M (\bm{\lambda}, J)\right) \prod_{i = 1}^{\max J^c} \frac{|l - \sigma_{i-1}|^{|l - \sigma_{i-1}|}}{(|l - \sigma_{i-1}| + 1)^{|l - \sigma_{i-1}| + 1}} \cdot \prod_{i = 1 + \max J^c}^n \frac{
        |\mu_i|^{|\mu_i|}}{(|\mu_i| + 1)^{|\mu_i| + 1}} 
    \end{align*}
where $\sigma_i = \sum_{k  =1}^i\mu_k$ and $l = |\bm{\lambda}|_1$. 
\end{corollary}
\begin{proof}
Follows immediately from \Cref{initial lower bound for parabolic verma modules} and \Cref{G_J is multipartite}.
\end{proof}
\subsection{Explicit bounds for the dimensions of weight spaces of parabolic Verma modules}
\label{subsections: explicit bounds for dimensions of weight spaces}
We further simplify the capacity term that shows up in \Cref{initial lower bound for parabolic verma modules} by restating $\cpc_{\mu} (\character M(\lambda,J))$ as a maximization problem. Evaluating this maximization problem at some point yields a lower bound  for $\cpc_{\mu}(\character M(\lambda,J)) $, which we use in \Cref{Secondary lower bound for parabolic verma modules} to obtain a more explicit bound for the dimension of weight spaces of parabolic Verma modules.

\begin{proposition}
\label{product of capacities}
For any $f_1, \cdots, f_k \in \R_{\ge 0}((x_1, \cdots, x_n))$ and $\bm{\alpha} \in \Z^n$, we have: 
\[
\cpc_{\bm{\alpha}}\prod_{i =1}^k f_i = \sup_{\overset{\bm{\alpha^i \in \R_{> 0}^n}}{\sum_{i = 1}^k \bm{\alpha^i} = \bm{\alpha}}} \prod_{i = 1}^k \cpc_{\bm{\alpha^i}} f_i
\]
\end{proposition}
\begin{proof}
    This proposition holds for all polynomials $f_1, \cdots, f_k \in \R_{\ge 0} [x_1, \cdots, x_n]$ by \cite[Proposition 6.2]{BLP23}. Moreover, \cite[Proposition 2.18]{GL21} states that if a sequence of polynomials $\{p_i\}_{i = 1}^{\infty}$ converges to some analytic function $p$ uniformly on compact sets, then $\cpc_{\bm{\alpha}} p = \lim_{i \to \infty} \cpc_{\bm{\alpha}} p_i$, completing the proof. 
\end{proof}
\begin{lemma}
    \label{capacity is at least the corresponfing coefficient} Let $p \in \R_{\ge 0}[x_1, \cdots, x_n]$ and let $\bm{\alpha} \in \Z_{\ge 0}^n$. Then:
    \[
        \cpc_{\bm{\alpha}} (p)  \ge p_{\bm{\alpha}}
    \]
\end{lemma}
\begin{proof}
    The statement follows directly from the fact that $\frac{p(\bm{x})}{\bm{x^{\alpha}}} \ge p_{\bm{\alpha}}$ for all $\bm{x}>\bm{0}$.
\end{proof}


\begin{corollary}
\label{Secondary lower bound for parabolic verma modules} Suppose  $M(\bm{\lambda}, J)$ is a parabolic Verma module. Fix the following:
\begin{itemize}
    \item $\bm{\nu} \in \Z^{[n+1]}_{\ge 0}$ with $\nu_{i_t} = 0$ and $\bm{\nu}_t := \restr{\bm{\nu}}{J_t}$ such that $|\bm{\nu}_t|_1 = |{\lambda}_t|_1$ for all $t \in \{0,1,\ldots,r\}$, and
    \item a (not necessarily integral) flow $\phi$ of $G_J$ with net-flows $\bm\mu-\bm\nu$.
\end{itemize}
Then
    \begin{align*}
        \dim M(\bm{\lambda}, J)_{\bm{\mu}} &\ge   \prod_{e \in E(G_J)}\frac{(\phi(e) + 1)^{\phi(e) + 1}}{\phi(e)^{\phi(e)}} \cdot \prod_{ t = 0}^r K_{\lambda_t, \bm{\nu}_t} \cdot \prod_{i = 1}^{n}\frac{|m_i|^{|m_i|}}{(|m_i| + 1)^{|m_i| + 1}} 
    \end{align*}
    where $m_i = -\sum_{k \in C_i} \mu_k$, and $K_{\lambda_t, \bm{\nu}_t}$ is the Kostka number indexed by $\lambda_t, \bm{\nu}_t$. 
\end{corollary}
\begin{proof}
Let us rewrite $\cpc_{\bm{\mu}}\left(\character M (\bm \lambda, J)\right)$. By \Cref{product of capacities}:
\begin{align*}
    \cpc_{\bm{\mu}}\left(f_{G_J}. \prod_{t = 0}^{r} s_{\lambda_t}\right) &= \sup_{\bm{\beta} + \sum_{t = 0}^{r} \alpha_t = \bm{\mu}} \cpc_{\bm{\beta}} f_{G_J} . \prod_{t = 0}^{r} \cpc_{\bm{\beta_t}}s_{\lambda_t}\\
    & \ge \cpc_{\bm{\mu - \nu}} (f_{G_J}). \prod_{t = 0}^{r} \cpc_{\bm{\nu_t}}(s_{\lambda_t})\\
    & \ge \sup_{\psi \in K_{G_J}{(\bm{\mu - \nu}})} \left\{ \prod_{e \in E(G_J)}\frac{(\psi(e) + 1)^{\psi(e) + 1}}{\psi(e)^{\psi(e)}}\right\}. \prod_{t = 0}^{r + 1} K_{\lambda_t, \bm{\nu}_t} \\
    & \ge \prod_{e \in E(G_J)}\frac{(\phi(e) + 1)^{\phi(e) + 1}}{\phi(e)^{\phi(e)}} \cdot \prod_{t = 0}^{r} K_{\lambda_r, \bm{\nu}_r},
\end{align*}
the last inequality is due to \Cref{convex analysis of capacity of flow series} and \Cref{capacity is at least the corresponfing coefficient}. The corollary follows from \Cref{initial lower bound for parabolic verma modules}.
\end{proof}

\subsection*{Acknowledgements}

We would like to thank Apoorva Khare, Jacob Matherne, and Alejandro Morales for helpful and interesting conversations.
Both authors acknowledge the support of the Natural Sciences and Engineering Research Council of Canada (NSERC), [funding reference number RGPIN-2023-03726]. Cette recherche a \'et\'e partiellement financ\'ee par le Conseil de recherches en sciences naturelles et en g\'enie du Canada (CRSNG), [num\'ero de r\'ef\'erence RGPIN-2023-03726].

\bibliographystyle{amsalpha}
\bibliography{uw-ethesis}
\appendix
\section{Missing proofs}
\label{AppendixA}
For completeness, the proofs of results used without proof in previous sections are included in this appendix. 
\subsection{Ratio test}
We state and ratio test for general positive sequences, and prove a version of it for log-concave sequences. 
\begin{lemma}
    \label{ratio test}
    Let $\{ a_n\}_{n \ge 0}$ be a sequence with $a_i \in \R_{>0}$. The sum $\sum_{n \ge 0} a_n$ converges if $\lim_{n \to \infty} \frac{a_{n+1}}{a_n} < 1$, and diverges if $\lim_{n \to \infty} \frac{a_{n+1}}{a_n}> 1$.
\end{lemma}
\begin{lemma}
    \label{ratio test for log-concave sequences}
     Let $\{ a_n\}_{n \ge 0}$ be a log-concave sequence with $a_i \in \R_{>0}$. Then the sequence $\{\frac{a_{n+1}}{a_n}\}_{n \in \N}$ is decreasing and convergent, and the sum $\sum_{n \ge 0} a_n$ converges if and only if $\lim_{n \to \infty}\frac{a_{n+1}}{a_n}< 1$.
\end{lemma}
\begin{proof}
    Because $\frac{a_{n+1}}{a_n} \le \frac{a_n}{a_{n-1}}$ for a log-concave sequence $\{a_n\}_{n \in \N}$, the sequence $\{\frac{a_{n+1}}{a_n} \}_{n \in \N}$ is decreasing. So $\{\frac{a_{n+1}}{a_n} \}_{n \in \N}$ is a decreasing positive sequence, and converges to some $L \in \R_{>0}$.

    By \Cref{ratio test}, $\sum_{n \ge 0}a_n$ converges if $L < 1$, and diverges if $L > 1$. It remains to show that this sum diverges in case $L = 1$ as well. 

    Note that $\{\frac{a_{n+1}}{a_n} \}_{n \in \N}$ is a decreasing sequence, and if $L = 1$, none of the terms $\frac{a_{n+1}}{a_n}$ can be less than 1, or we will end up with a sequence that is eventually strictly less 1. So we should have $\frac{a_{n+1}}{a_n} \ge 1$ for all $n \in \N$, and $a_n \ge a_1$. Now we have: 
    \[
    \sum_{n \in \N}a_n \ge \sum_{n \in \N}a_1
    \]
    and this sum diverges.
\end{proof}

\subsection{Proof of \Cref{moving the boundary case in to the interior} for $n = 3$}
\begin{lemma}
\label{base case for induction proving that domain of convergence of two Laurent series is equal}
Let 
\[
p(x_1, x_2, x_3) = x_3 r(x_1, x_2) + s(x_1, x_2)
\]
be a $(d+1 )$-homogeneous DL Laurent series. Assuming that $r$ and $s$ are both non-zero Laurent series, we have    $(1,1) \in \Omega_r $ if and only if $ (1,1) \in \Omega_s$.
\end{lemma}
\begin{proof}
Write: 
 \begin{align*}
        r(x_1, x_2) &= \sum_{n \in \Z} r_n x_1^n x_2^{-n + d}, \\
        s(x_1, x_2) &= \sum_{n \in \Z} s_n x_1^n x_2^{-n + d + 1},
    \end{align*}
and let $D_s:=\Deg_{x_1}(s) = \{n: s_n \neq 0 \}$, $D_r:=\Deg_{x_1}(r) = \{ n: r_n \neq 0\}$, and $D := D_r \cap D_s$. We need to show that that $\sum_{n \in D_r} r_n$ converges if and only if $\sum_{n \in D_s}s_n$ converges. $s$ and $r$ are two bivariate DL Laurent series by \Cref{the coefficient of x^d in a DL Laurent series is DL}, and therefore, the sequences of their coefficients are log-concave by \Cref{bivaraite Laurent series is DL iff it has a log concave sequence}. Particularly, $\{r_n\}_{n \in \Z}$ and $\{s_n\}_{n \in \Z}$ cannot have internal zeros, implying that $D_r$ and $D_s$ are two sets of consecutive integers. We will prove that $D_r \Delta D_s$ has finitely many elements. 

Take an arbitrary $n \in D_s$ and an arbitrary element $m \in D_r$. So we have $\bm{\nu}:= (n, -n + d + 1, 0) \in \supp(p)$ and $\bm{\mu}:= (m, -m + d, 1) \in \supp(p)$. Note that the support of $p$ is M-convex by \Cref{support of DL Laurent series is M convex}, and we have two elements of $p$'s support differing in their third coordinate, implying that there should exist some $i \in \{1,2\}$ satisfying $\bm{\mu} - \bm{e}_3 + \bm{e}_i, \bm{\nu} + \bm{e}_3 -\bm{e}_i \in \supp(p)$. If $i = 1$, we obtain $(m + 1 , -m + d, 0), (n-1, -n + d + 1, 1) \in \supp(p)$, and therefore $m+1 \in D_s$ and $n-1 \in D_r$. In case $i = 2$, we get $(m, -m + d + 1, 0), (n, -n + d, 1) \in \supp (p)$, so $m \in D_s$ and $n \in D_r$. Thus, we have proven so far that for any element $n \in D_s$, either $n$ itself or $n-1$ is an element of $D_r$ as well, and conversely, for any $m \in D_r$, either $m$ or $m+1$ is in $D_s$. So we have: 
\[
\min(D_r) \ge \min(D_s) - 1 \hspace{0.2 cm} \text{and} \hspace{.2 cm} \max(D_s) \le \max(D_r) + 1,
\]
where $\min(A), \max(A)$ are the minimum and the maximum elements of a finite set $A$ respectively. Taking into consideration that both $D_r$ and $D_s$ are sets of consecutive integers, we conclude $|D \Delta D_s|, |D \Delta D_r| \le 1$.  We now just have to prove that the sum $\sum_{n \in D} s_n$ converges if and only if $\sum_{n \in D}r_n$ converges. 

For any $n \in D$, define:
    \begin{align*}
         \epsilon_n &= \frac{r_{n+1}}{r_n}, \\
         \delta_n &= \frac{s_{n+1}}{s_n},
    \end{align*}

The idea is to prove that $\epsilon_n = o(\delta_n)$, and that the coefficients of neither $s$ nor $r$ grow faster than those of the other. Since $p$ is a DL Laurent series, for any $k \in \Z$, the polynomial:
    \[
    \Poly[x_1^{-k} x_2^{k - d + 1} p] = s_k x_2^2 + s_{k + 1}x_1x_2 + s_{k + 2} x_1^2 + r_k x_3 x_2 + r_{k + 1} x_3 x_1
    \]
is DL, and the matrix:
    $$\nabla^2 N[\Poly[x_1^{-k} x_2^{k - d + 1} p]] =  \begin{bmatrix}
        0 & r_k & r_{k + 1} \\
        r_k & s_k & s_{k + 1}\\
        r_{k + 1} & s_{k + 1} & s_{k + 2}
    \end{bmatrix}$$
    has at most one positive eigen-value by \Cref{definition of Lorentzian polynomials} (this hessian is written in the reverse order of variables). So its determinant must be nonnegative. If $k + 1 \in D$, we can divide the first row and the first column of this hessian by $r_{k+1}$, to get a matrix whose determinant is still nonnegative, giving us: 
    \begin{align*}
    0 &\le \begin{vmatrix}
        0 & 1 & \epsilon_k \\
        1 & s_k & \delta_k s_k \\
        \epsilon_k & \delta_k s_k & \delta_k \delta_{k + 1} s_k
    \end{vmatrix}\\
    &= -\epsilon_k^2 s_k + 2 \epsilon_k \delta_k s_k - \delta_k \delta_{k+ 1}s_k ,
    \end{align*}
    which is a quadratic form in $\epsilon_k$, and as long as $s_k > 0$ (so for instance if $k \in D$), we will get:
    \[ \delta_k - \sqrt{\delta_k^2 - \delta_k \delta_{k + 1}} \le \epsilon_k \le \delta_k + \sqrt{\delta_k^2 - \delta_k \delta_{k + 1}}. \]

    First, suppose $(1,1) \in \Omega_s$, and let us prove that $\sum_{n \in D_{\ge 0}} r_n$ and $\sum_{n \in D_{< 0}}r_n$ both converge, where $D_{\ge 0} = D \cap \Z_{\ge 0}$ and $D_{< 0} = D \cap \Z_{< 0}$.
    
    If $D_{\ge 0}$ is a finite set, $\sum_{n \in D_{\ge 0}}r_n$ converges without any further issues. So suppose $D_{\ge 0}$ is an infinite set. Recall that $D$ is a set of consecutive integers, and therefore, $D_{\ge 0}$ contains all nonngeative integers greater than some $M \in \Z_{\ge 0}$. We conclude that $\delta_n$ is a well defined ratio for all $n \ge M$. Apply \Cref{ratio test for log-concave sequences} to $\sum_{n \ge M} s_n$ to obtain $\delta^* := \lim_{n \to \infty}\delta_n < 1$. 

    Let $k_0 \in \N$ be a large enough integer satisfying $k_0, k_0 + 1 \in D$ and $\delta_{k_0} < 1$, and let $c = {\frac{ (1- \delta_{k_0})^2}{2{\delta_{k_0}}}}$. Now, there exists some large enough $k \ge k_0 $ satisfying $\delta_k - \delta_{k + 1} \le c$, therefore:
    \[
    \epsilon_k \le \delta_k + \sqrt{\delta_k}. \sqrt{ \delta_k - \delta_{k + 1}} \le \delta_{k_0} + \sqrt{\delta_{k_0}}.\sqrt{c} < 1
    \]
    So $\lim_{n \to \infty} \epsilon_n < 1$, and $\sum_{n \in D_{\ge 0}}r_n$ converges by \Cref{ratio test for log-concave sequences}. 

    If $D_{<0}$ is finite, we don't have anything to prove. So we can assume that $D_{<0}$ is the set of all negative integers smaller than some $M \in \Z$, and that $\delta_{-n}$ is a well defined ratio for all $n > |M|$. Apply \Cref{ratio test for log-concave sequences} to get $\delta_*^{-1} :=\lim_{n \to \infty}\delta_{-n}^{-1} < 1$. 
    
    If $\delta_* \neq \infty$, we have that $\lim_{n \to \infty} \delta_{-n} = \delta_* > 1$.
    Let $k_1$ be a large enough integer satisfying $\delta_{-k_1}>1$, and let $c = \frac{(\delta_{-k_1}-1)^2}{2\delta_*}$. There exists some large enough $k \ge k_1$ satisfying $\delta_{-k} - \delta_* \le c$ and $\delta_{-k} - \delta_{-k+1} \le c$. Note that $\{\delta_{-n}\}_{n \in \N}$ is an increasing sequence and we have: 
    \begin{align*}
        \epsilon_{-k} &\ge \delta_{-k} - \sqrt{\delta_{-k}}. \sqrt{(\delta_{-k} - \delta_{-k + 1})}
        \ge \delta_{-k_1} - \sqrt{\delta_*}. \sqrt{c} > 1
    \end{align*}
    So $\{\delta_{-n}\}_{n \in \N}$ is an increasing sequence with $\delta_k > 1$, and therefore $\lim_{n \to \infty} \delta_{-n}^{-1} < 1$.

    On the other hand, let $\delta_*=\infty$, and assume for the sake of contradiction that $\epsilon_{-k}$ is always less than 1 for $k \in \N$. We know that $\delta_{-k} - \sqrt{\delta_{-k}^2 - \delta_{-k} \delta_{-k+1}} \le \epsilon_{-k} \le  1$, and we can rewrite this inequality to get:
    \begin{align*}
        \delta_{-k}^2 - \delta_{-k} \delta_{-k+1} &\ge (\delta_{-k} - 1)^2 \\
        &= \delta_{-k}^2 - 2\delta_{-k} + 1 \\
        \implies \delta_{-k}(2 - \delta_{-k + 1}) &\ge 1
    \end{align*}
    and therefore, $\delta_{-k + 1} \le 2$, a contradiction since we know $\lim_{n \to \infty} \delta_{-n} = \infty$. Note that we could assume that the sequence $\{ \epsilon_{-k} \}_{k \in \N}$ is bounded above by any number and derive a similar contradiction. So in this case, not only $\lim_{n \to \infty} \epsilon_{-n}$ is strictly larger than 1, but it is infinite as well. 

    Overall, we have proven that $\lim_{n \to \infty} \epsilon_{-n}^{-1} < 1$, and $\sum_{n \in D_{<0}}r_n$ converges by \Cref{ratio test for log-concave sequences}.
    
    For the second part of the proof, assume that $(1,1) \in \Omega_r$. We need to show that both of the sums $\sum_{n \in D_{\ge 0}} s_n$ and $\sum_{n \in D_{<0}} s_n$ converge. 

    Suppose $D_{\ge 0}$ is the set of all nonnegative integers greater than $M$, or we have nothing to prove. By \Cref{ratio test for log-concave sequences}, the limit of $\epsilon_n$ as $n$ goes to infinity be some $\epsilon^*<1$. 

    Let $m_0$ be a large enough integer satisfying $\epsilon_{m_0} < 1$, and let $c = \frac{(1 - \epsilon_{m_0})^2}{2\delta_{m_0}}$. Note that the sequence $\{\delta_k\}_{k > M }$ converges by \Cref{ratio test for log-concave sequences}, and therefore, there exists some large enough $m$ satisfying $\delta_{m} - \delta_{m + 1} \le c$, therefore: 

    \begin{align*}
         \delta_m & \le \epsilon_m + \sqrt{\delta_m} \sqrt{\delta_m - \delta_{m + 1}} \\
         & \le \epsilon_{m_0} + \sqrt{\delta_m} \sqrt{\delta_m - \delta_{m + 1}} \\
         & \le \epsilon_{m_0} + \sqrt{\delta_{m_0}}\sqrt{c} < 1        
     \end{align*}
    
  For the last part of this proof, assume that $D_{< 0} = \Z_{<0}$. The ratio test for the convergence of $\sum_{n \in \N} r_{-n}$ tells as that $\lim_{n \to \infty} \epsilon_{-n}^{-1} \le 1$. If this limit is exactly 1, we will get a contradiction, so let $ \epsilon_*^{-1} = \lim_{n \to \infty} \epsilon_{-n}^{-1} < 1$. 

  For the last part of this proof, suppose $D_{<0}$ is an infinite set and contains all negative integers smaller than some $M \in \Z$. Then $\epsilon_{-n}$ is a well-defined ratio for all $n > |M|$, and by \Cref{ratio test for log-concave sequences}, $\epsilon_*^{-1} = \lim_{n \to \infty} \epsilon_{-n}^{-1} < 1$.

  If $\epsilon_* = \infty$, we have:
  \[
  \epsilon_{-k} \le \delta_{-k} + \sqrt{\delta_{-k}^2 -\delta_{-k} \delta_{-k + 1}} \le 2 \delta_{-k},
  \]
  and $\{ \delta_{-n}\}_{n > |M|}$ goes to infinity, in which case $\lim_{n \to \infty} \delta_{-n}^{-1} < 1$ and $\sum_{n \in D_{< 0}} s_n$ converges by \Cref{ratio test for log-concave sequences}.

So assume that $1 < \epsilon_* < \infty$. Moreover, if $\delta_*^{-1} := \lim_{n \to \infty}\delta_{-n}^{-1}$, we can assume $\delta_* < \infty$ as well (note that $\{ \delta_{-n}^{-1}\}_{n > |M|}$ converges by \Cref{ratio test for log-concave sequences}).  Let $m_1$ be a large enough integer satisfying $\epsilon_{-m_1} > 1$, and let $c = \frac{(\epsilon_{-m_1} - 1)^2}{2 \delta_*}$. There exists some large enough $m$ such that $\delta_{-m} - \delta_{-m + 1} \le c$, and we will get: 
\begin{align*}
    \delta_{-m} &\ge \epsilon_{-m} - \sqrt{\delta_{-m}}. \sqrt{\delta_{-m} - \delta_{-m + 1}}\\
    & \ge \epsilon_{-m_1} - \sqrt{\delta_*}. \sqrt{c}\\
    & > 1
\end{align*}
So $\delta_{-m} > 1$ for some $m > |M|$, and $\lim_{n \to \infty} \delta_{-n}^{-1} < 1$. So $\sum_{n \in D_{<0}}s_n$ converges by \Cref{ratio test for log-concave sequences}, completeing the proof. 
\end{proof} 
\end{document}